\documentclass[preprint]{imsart}

\usepackage{amsthm,amsmath,amstext}
\usepackage{amsfonts,amssymb}
\usepackage{bm}
\usepackage{graphicx}
\usepackage{xcolor}
\RequirePackage{natbib}

\startlocaldefs
\DeclareMathOperator{\tr}{tr}
\newcommand{\R}{\ensuremath{\mathbb{R}}}

\newcommand{\E}{\mathbb{E}}
\newcommand{\x}{\bm x}
\newcommand{\y}{\bm y}
\newcommand{\w}{\bm w}
\newcommand{\Sp}{\bm S_p}
\newcommand{\F}{\bm F}
\newcommand{\bI}{\bm I}
\newcommand{\bS}{\bm{\hat\Sigma}}
\newcommand{\bSa}{\bm{\hat\Sigma}(\alpha)}

 \newcommand\be{\begin{equation}} \newcommand\ee{\end{equation}}
\newcommand\bea{\begin{eqnarray}}
\newcommand\eea{\end{eqnarray}}

\newcommand{\on}{{i=1}^n}

\makeatletter
\newcommand{\vast}{\bBigg@{4}}
\newcommand{\Vast}{\bBigg@{5}}
\makeatother

\newtheorem{thm}{Theorem}

\newtheorem{lem}{Lemma}
\newtheorem{prop}{Proposition}

\newtheorem{rek}{Remark}

\endlocaldefs

\usepackage{makeidx}
\makeindex

\begin{document}

\begin{frontmatter}

\title{Robust Sparse Covariance Estimation by Thresholding Tyler's M-Estimator}
\runtitle{Robust Sparse Covariance Estimation}

\begin{aug}
\author{\fnms{John} \snm{Goes}\ead[label=e1]{johngoes@umn.edu}}
\address{University of Minnesota\\
School of Mathematics \\
\printead{e1}}
\author{\fnms{Gilad} \snm{Lerman}\ead[label=e2]{lerman@umn.edu}}
\address{University of Minnesota\\
School of Mathematics \\
\printead{e2}}
\author{\fnms{Boaz} \snm{Nadler}\ead[label=e3]{boaz.nadler@weizmann.ac.il}}
\address{Weizmann Institute of Science\\
Department of Computer Science and Applied Mathematics \\
\printead{e3}}

%

\affiliation{
University of Minnesota
and
Weizmann Institute of Science
}

\runauthor{Goes, Lerman and Nadler}
\end{aug}

\begin{abstract}
    Estimating a high-dimensional sparse covariance matrix
from a limited number of samples is a fundamental problem in contemporary
data analysis. Most proposals to date, however, are not robust to outliers
or heavy tails. Towards bridging this gap, in this work we consider estimating a sparse shape matrix
from $n$ samples following a possibly heavy tailed elliptical distribution. We propose estimators based on thresholding either Tyler's M-estimator or its regularized
variant.
We derive bounds on the difference in spectral norm between our estimators and the shape matrix in the joint
limit as the dimension $p$ and sample size $n$ tend to infinity with
$p/n\to\gamma>0$. These bounds are minimax rate-optimal.
Results on simulated data support our theoretical analysis.
\end{abstract}





\end{frontmatter}

\section{Introduction}


The covariance matrix \(\bm \Sigma \) of a \(p\)-dimensional random variable $X$ is a central object in statistical data analysis. Given $n$ observations $\{\bm{x}_i\}_{i=1}^n$, accurately estimating this matrix is of great importance for many tasks including PCA, clustering and discriminant analysis \citep{anderson,mkb79}.
The
sample covariance matrix, which is the standard estimator for \(\bm\Sigma\),  is quite accurate when the random variable \(X\) is sub-Gaussian and \(p\ll n \).

In several contemporary applications, however, the number of samples \(n\)
and the dimension \(p\) are comparable,
and the data may be heavy tailed.
To accurately estimate the covariance matrix when  $n$ and $p$ are comparable, additional assumptions, such as its approximate sparsity are typically made. Over
 the past decade several sparse covariance matrix estimators were proposed and
 analyzed \citep{CT,cai2011adaptive,karoui,lam2009sparsistency,rothman2009generalized}. In
 addition, minimax lower bounds for estimating sparse covariance matrices in
 high-dimensional settings were established
 \citep{minimax,cai2012minimax_l1,cai2016estimating}.

With respect to heavy tailed data, a popular model which we consider in this
work is the  elliptical distribution
\citep{cambanis1981,fang1990,frahmthesis,kelker1970}.
An elliptical distribution is characterized by a $p\times p$ shape or scatter matrix \(\Sp\), which equals a multiple of its population covariance matrix, when the latter exists. Since an elliptical distribution may be heavy tailed,
the classical sample covariance may exhibit large variance and be a poor estimator of the population covariance \citep{notefficient}. Moreover, the elliptical distribution might be so heavy tailed as to not even have finite second moments, in which case its population covariance does not exist.
Yet due to the structure of the elliptical distribution, even with heavy tails it is
nonetheless possible to accurately estimate its shape matrix. This is useful in various
applications, since the shape matrix preserves the directional properties of the
distribution, such as its principal components.

Following Huber's pioneering work \citep{Huber_book}, over the past decades several
robust estimators of the covariance and
shape matrix were proposed and theoretically studied, see  \citet{maronna1976,maronna2017robust,kent,dumbgen2015,dumbgen2016} and references therein. For elliptical distributions, \citet{Tyler} proposed a robust M-estimator
for the scatter matrix \(\Sp\) and an iterative
scheme to compute it.
Tyler's M-estimator has found widespread use in various applications involving heavy tailed data. However, as it is defined only for \(p<n\), in recent years
several regularized variants, applicable also for \(p>n\) were proposed and
analyzed~\citep{abram,wiesel,chen,sun,pascal,ollila2014}. The spectral 
properties of Maronna's M-estimators and specifically Tyler's M-estimator and its regularized variants, in high dimensions as $n,p\to\infty$ with $p/n\to\gamma$ were studied by \citet{dumbgen_tylers,couillet2014robust,couillet2015random,MP,MP1,MP2}, among others.
For a recent survey on Tyler's M-estimator and its variants, see \citet{wieselteng2014}.

In this paper we study the combination of heavy tailed data with a ``large
\(p\) -- large \(n\)" setting. As formulated in Section~\ref{Sec:problem}, we consider robust
estimation of the shape matrix of an  elliptical distribution, assuming
it is approximately sparse. We address the following two challenges:\ (i) design a
computationally efficient and statistically accurate estimator of the shape
matrix \(\bm S_p\), that is adaptive to its unknown sparsity parameters; (ii)
provide theoretical guarantees on its accuracy, in the large \(p\) large $n$ regime.

We make the following contributions. First, in Section
\ref{Sec:main} we propose simple and computationally efficient estimators for the
sparse shape matrix of an elliptical distribution. These are based on
thresholding either Tyler's M-estimator (TME) or its regularized variant. Second, we
provide theoretical guarantees on their accuracy in the limit \(n,p\to\infty\) with
$p/n\to\gamma$.
Theorems \ref{plessn} and \ref{regthm} show that
the estimator $\hat{\bm E}$ based on  thresholding either TME for \(\gamma<1\) or its regularized variant for any $\gamma\in(0,\infty)$, converges in spectral norm to a sparse shape matrix \(\Sp\) at rate $\|\hat{\bm
E}-\Sp\|=O_P((\log p/n)^{(1-q)/2})$, where $q$ is the sparsity parameter of $\Sp$. Estimating a sparse shape matrix under a heavy tailed elliptical distribution is thus possible with the same asymptotic error rate as estimating a sparse covariance matrix under sub-Gaussian distributions. Moreover, our estimators are rate optimal, as this rate coincides with the
minimax rate for sparse covariance estimation with sub-Gaussian data \citep{minimax}\footnote{Technically the minimax rate was proven under the assumption that
$p/n^\beta\to c$ with $\beta>1$, see Remark 5 in \cite{minimax}. However, from personal communication with Profs. Cai and Zhou, the same minimax rate should hold also when $\beta=1$.}.

Our proofs follow the approach of \cite{CT}, with required modifications given that we analyze Tyler's M-estimators.
Theorem~\ref{plessn}, which provides guarantees for TME and is thus valid for \(p<n\), is
proven in Section~\ref{unreg}. The proof is relatively simple and heavily relies on  \citet{MP}, who
studied the spectral properties of Tyler's M-estimator when $n,p\to\infty$. Theorem~\ref{regthm}  provides guarantees on the
thresholded regularized TME, and is thus applicable also for \(p>n\). As detailed in Section \ref{reg},  its proof is far more involved, and combines a careful analysis of the form of the regularized TME together with several results in random matrix theory.
Section~\ref{simulations}  presents simulation results that support our
theoretical analysis.
With an eye towards practitioners, given that
regularization is common also when \(p<n\), we focus  on the regularized TME.
With Gaussian data, our thresholded TME
estimator is as accurate as thresholding the sample covariance. In contrast, in the presence of heavy tails it is far more
accurate. We also illustrate its potential utility in handling outliers. In addition, our estimator is quite fast to compute in practice, requiring only
few seconds on a standard PC, say for \(p=n=1000\).

Our work is related to several recent papers, that also considered sparse shape
or covariance matrix estimation with heavy tailed data.  \citet{robscatter}
considered a pair-elliptical distribution, which is a different generalization
of the classical elliptical distribution from the one we consider. They assumed
moderate tails so the population covariance matrix exists, and proposed an
estimator for it. They provided finite sample approximation bounds for their
estimator, which depend on various properties of the distribution.  For well-behaved
elliptical distributions with an exactly sparse covariance matrix, their estimator
is minimax rate optimal under the Frobenius norm.
\citet{tylercov} considered estimating a covariance matrix from a convex subset
of all positive semidefinite matrices. They added
a convex regularization term to the TME and solved the resulting
optimization problem by a semidefinite program (SDP). They
proved the existence of their estimator and its asymptotic consistency
for fixed dimension $p$ and $n\to\infty$. However, their SDP-based method is  computationally demanding
even for moderate values of $n$ and $p$.
\citet{sun2016robust} considered a wider non-convex class of matrices, and derived an SDP-based algorithm with lower time complexity.

\citet{matrixdepth} considered an elliptical distribution,  corrupted by an
epsilon-contamination model. They proposed several
estimators for the shape matrix of the elliptical distribution, based
on a generalization of Tukey's depth function.  Under a notion of
sparsity different from the one considered here,
they proved their estimator is minimax rate optimal when
$n,p\to\infty$ and $(\log p) / n\to0$.  However, from a
practical perspective this depth function estimator has a significant limitation
-- it is intractable to compute.
\cite{balakrishnan2017computationally} considered an epsilon-contamination model for a Gaussian distribution with
sparse covariance matrix ${\bm \Sigma}$, such that $\|{\bm \Sigma - \bm I}\|_0 \leq s$ for a fixed $s \geq 0$.
They proposed a polynomial-time algorithm for robust covariance estimation under this model and established an upper bound on its error under  Frobenius norm, assuming  $n,p\to\infty$ and $(\log p) / n \to c \geq 0$.
Our work in contrast provides a computationally efficient and rate optimal estimator for an approximately sparse shape matrix of a potentially heavy tailed elliptical distribution in the high dimensional setting \(p,n\to\infty \) with $p/n\to \gamma$. 
Finally, \cite{avella2018robust} developed rate optimal 
robust sparse covariance estimators for heavy tailed distributions 
via a different approach than the one presented here, based on various robust pilot estimators. Further discussion and directions for future research appear in Section \ref{Sec:disc}.



\section{Problem Setting}\label{Sec:problem}

With precise definitions  below, given $n$ i.i.d.~observations from an elliptical
distribution, the problem we study is how to estimate  its $p\times p$
\emph{shape matrix} $\bm S_p$.
Of particular interest to us is the high-dimensional regime, where both \(p,n\)
are large and comparable. Following previous works, to be able to accurately estimate the shape matrix in this regime we assume that it is approximately sparse. For completeness, we first introduce some notation, briefly review the elliptical distribution and the class of approximately sparse shape matrices we consider.


\paragraph{Notation} We denote vectors by bold lowercase letters as in \(\bm v\), and matrices by bold uppercase letters as in $\bm A$. For a vector \(\bm v\in\mathbb{R}^n\),  $\|\bm v\|$ is its Euclidean norm, $\|\bm v\|_\infty=\max_i |v_i|$, and  $B_{R}(\bm u)=\{\bm v\in\mathbb{R}^n\, |\, \|\bm v-\bm u\|_{\infty}\leq R\}$. The unit sphere in $\mathbb{R}^p$ is denoted $\mathbb{S}^{p-1}$. The identity matrix is $\bI$ and $\bm 0$ and $\bm 1$ are the vectors of zeros and ones respectively, with dimensions clear from the context. For a matrix $\bm A=(a_{ij})$, $\|\bm A\|$ denotes its spectral norm, $\|\bm A\|_F$ its Frobenius norm, $\|\bm{A}\|_{\max} = \max_{i,j} |a_{ij} |$ and $\|\bm A\|_\infty=\max_{i}\sum_j |a_{ij}|$.
We denote the set of $p \times p$ symmetric positive semidefinite and definite matrices by $S_{+}^p$ and $S_{++}^p$ respectively.
We say that an event occurs with high probability (abbreviated w.h.p.), if its probability is at least $1- C\exp(-cp)$ for constants $c,C>0$ independent of $p$.

\paragraph{Elliptical Distribution and its Shape Matrix}
{A random vector $\bm{x}\in\R^p$ follows an elliptical 
distribution with location vector $\bm \mu$ if
it has the form
\bea
        \bm{x}=\bm \mu + u \bm S_p^{\frac 12} \bm{\xi} = \bm \mu + u \bm{z},
                \label{data}
\eea
where $\bm\xi$ is drawn uniformly from $\mathbb S^{p-1}$,  $\bm S_p \in S_{++}^p$, and $u$ is an arbitrary random or deterministic nonzero scalar,  independent of $\bm{\xi}$.

In Eq.~(\ref{data}),  $\bm S_p$ is not unique, as it can be arbitrarily scaled with $u$ absorbing the
inverse scaling factor. Without loss of generality, we thus fix
\bea
\tr(\bm S_p)=p,
    \nonumber \label{tr1scale}
\eea
and refer to $\bm S_p$ as the \textit{shape matrix}. This normalization is natural in the sense that the mean variance of the $p$ coordinates of $\bm z$ is one. If the
distribution is elliptical and the population covariance $\bm \Sigma$ exists, then  \(\bm \Sigma= c\bm S_p\) for some constant \(c>0\), see for example \citet{tylercov}.

An important property of the elliptical distribution is that if $\x_1,\x_2$ are independent random vectors from (\ref{data}), then $\x_1-\x_2$ has an elliptical distribution with the same shape matrix $\Sp$ but with a zero location vector $\bm \mu = \bm 0$.
When the goal is to estimate the shape matrix $\Sp$, this allows to remove the typically unknown location vector by a symmetrization principle \citep{dumbgen_tylers}. Specifically, all $\x_i-\x_j$
are elliptically distributed with location vector $\bm\mu=\bm 0$, and one may estimate the shape matrix using all of these pairwise differences   \citep{dumbgen_tylers,sirkia2007symmetrised}. 
As discussed by \citet{nordhausen2015cautionary}, such a procedure is beneficial also for non-elliptical distributions. The resulting $O(n^2)$ pairs are, however, dependent which may complicate the analysis of the resulting estimator. 
For simplicity, we shall thus assume to have
initially observed $2n$ i.i.d. samples $\tilde\x_1,\ldots,\tilde\x_{2n}$ from model (\ref{data})
and in what follows consider the $n$ differences $\x_i = \tilde \x_{2i}-\tilde \x_{2i-1}$ for $i=1,\ldots,n$ which form an i.i.d. sample from the elliptical distribution (\ref{data}) with location vector $\bm \mu=\bm 0$.

\paragraph{Approximate Sparsity of the Shape Matrix}
Following \citet{CT},
we consider the following class of row/column
approximately sparse  shape matrices with fixed parameters $0\leq q\le 1, M>0$ and $s_p>0$:
\begin{align}
&\mathcal{U}(q,s_p,M)  \nonumber 
= 
\Big\{ \bm A  \in S_{++}^p :    a_{ii}\le M, \ \sum_{j=1}^p |a_{ij} |^q \le s_p, \  1\leq i \leq p
\Big\}.
    \nonumber
\end{align}

\paragraph{Problem Statement}
Let $\{\bm x_i\}_{i=1}^n$
be \(n\) i.i.d.~samples from the model (\ref{data}) with location vector $\bm \mu = \bm 0 $ and a sparse
shape matrix  $\bm{S}_p \in\mathcal{U}(q,s_p,M).$ We consider the following two problems:\ (i) Without explicit knowledge of  $q,s_p$ and $M$, design a computationally efficient and statistically accurate estimator of the shape matrix \(\bm S_p\); (ii) Provide theoretical guarantees on its accuracy, in the asymptotic limit as \(p,n\to\infty\) with $p/n\to\gamma\in(0,\infty)$.

\section{Sparse Shape Matrix Estimation}\label{Sec:main}

If the elliptical distribution is sub-Gaussian, then thresholding the sample covariance matrix, proposed by  \cite{CT} and \cite{karoui}, yields an accurate estimate of \(\Sp \) up to a multiplicative scaling. As illustrated in Section \ref{simulations}, however, in the presence of heavy tails, the individual entries of the sample covariance matrix may be quite far from their population counterparts, and thresholding them may give a poor estimate of the shape matrix.

To handle heavy tails, we propose the following approach: compute Tyler's M-estimator (TME) or its regularized variant, and threshold it. In Section \ref{Sec:tyler} we review TME and its regularized variant. We prove that computing the latter is computationally efficient. Section \ref{Sec:Threshold_TME}
 presents our proposed estimators. A theoretical analysis of their accuracy appears in Section \ref{sec:Accuracy_Thresholded_TME}.

\subsection{TME and its Regularized Variant}  \label{Sec:tyler}

TME, proposed by \citet{Tyler} for elliptical distributions with a known location vector, which w.l.o.g. is assumed to be $\bm 0$,
is a \(p\times p\) matrix $\hat{\bm \Sigma}$ which satisfies
\bea
\frac{p}{n}\sum_{i=1}^n \frac{\x_i\x_i^T}{\x_i^T
\hat{\bm{\Sigma}}^{-1}\x_i} = \hat{\bm{\Sigma}}.
    \label{Tyler}
\eea
Here, samples $\x_i$ lying at the origin are ignored as they provide no information on the scatter matrix, and $n$ is the number of samples not at the origin. 
As solutions to (\ref{Tyler}) can be multiplied by an arbitrary constant, \citet{Tyler} considered the normalization $\tr(\hat{\bm\Sigma})=p$, and suggested to solve  Eq.~(\ref{Tyler}) by the following iterations, starting from an arbitrary $\hat{\bm\Sigma}_1 \in S_{++}^p$,
\bea
\hat{\bm{\Sigma}}_{k+1} =p \sum_{i=1}^n\frac{\bm{x}_i\bm{x}_i^T}{\bm{x}_i^T\hat{\bm{\Sigma}}_{k}^{-1}\bm{x}_i}\bigg/
\tr\Big(\sum_{i=1}^n\frac{\bm{x}_i\bm{x}_i^T}{\bm{x}_i^T\hat{\bm{\Sigma}}_{k}^{-1}\bm{x}_i}
\Big).
    \nonumber
\eea

\citet{kent1988maximum}[Theorems 1 and 2] showed that if any linear subspace in
$\R^p$ of dimension $1 \leq d \leq p-1$ contains less than $nd/p$ of the data
samples, then there exists a unique solution
to Eq.~(\ref{Tyler}), and the above iterations converge to it. With $n$ i.i.d.~observations from an elliptical distribution, no samples lie at the origin and this
condition holds with probability $1$.

TME enjoys several important properties: First, it may equivalently be defined as the
minimizer of the following cost function, over all positive definite matrices with the constraint $\tr(\bm R)=p$, 
\be
L(\bm R)=\frac{p}n\sum_{i=1}^n \log\left(\x_i^T\bm R^{-1}\x_i\right) + \log(\det(\bm R)).
        \label{eq:Loss_TME} 
\ee
 As the minimizer of Eq. (\ref{eq:Loss_TME}), $\hat{\bm\Sigma}$ is thus the maximum likelihood estimator of the shape matrix of both the angular central Gaussian distribution \citep{tyler1987statistical} and of the generalized elliptical
distribution~\citep{frahm2}. Moreover, it is the
``most robust" estimator of the shape matrix with fixed $p$ and $n\to\infty$ for
data i.i.d.~from a continuous elliptical distribution \citep[Remark 3.1]{Tyler}. TME outperforms the sample
covariance in a variety of applications, including finance~\citep{finance},
anomaly detection in wireless sensor networks~\citep{chen}, antenna array processing~\citep{antenna} and radar detection~\citep{radar}.

As the TME does not exist when $p>n$, several regularized variants have been
proposed and analyzed \citep{abram,chen,wiesel,pascal,sun}. Even when $p \leq n$, it is common
to add small regularization to the TME.
Following \cite{sun}, here we use a regularization parameter $\alpha>0$ and consider the following regularized TME $\hat{\bm\Sigma}(\alpha)$,   defined as the solution of
\bea
\hat{\bm \Sigma}(\alpha)=\frac{1}{1+\alpha}\frac{p}{n}\sum_\on \frac{
\x_i\x_i^T}{\x_i^T \hat{\bm \Sigma}(\alpha)^{-1}
\x_i}+\frac{\alpha}{1+\alpha}\bm I.
    \label{regTME}
\eea
If $\alpha=0$, Eq.~\eqref{regTME} reverts to Eq.~\eqref{Tyler}. While regularization towards general target matrices is possible \citep{wiesel}, here for simplicity
we consider only regularization towards the identity. In contrast to the original TME formulated in Eq. (\ref{Tyler}), for which the solution
can be multiplied by an arbitrary positive scalar, as proven by \citet[Proposition III.1]{pascal}, any solution to Eq. (\ref{regTME}) satisfies
$\tr(\hat{\bm \Sigma}(\alpha)^{-1}) = p$, regardless of the  value of $\alpha$. 

\citet[Theorem 11 and Proposition 13]{sun}, derived a sufficient and necessary condition for existence of a unique positive definite matrix which solves  Eq. (\ref{regTME}). Again, ignoring samples at the origin, the condition is that any linear subspace in
    $\R^p$ of dimension $1 \leq d \leq p-1$ contains less than $(1+\alpha)nd/p$
    of the data samples. Since $\alpha > 0$, this condition is weaker than for the original TME. In particular, with data 
    i.i.d. from a continuous distribution,
    Eq.~\eqref{regTME} has a unique solution for  $\alpha>\max(0,p/n-1)$, see also \citet[Theorem III.1]{pascal}.
With $n$ i.i.d.~samples from an elliptical
distribution, these  conditions hold with probability 1.  

\citet[Proposition 18]{sun}
further showed that starting from any positive definite initial
guess, the following iterations
\bea
\hat{\bm{\Sigma}}_{k+1}(\alpha)=\frac{1}{1+\alpha}\frac{p}{n}\sum_\on
\frac{\x_i\x_i^T}{\x_i^T {{\hat{\bm{\Sigma}}_k}(\alpha
)}^{\!-1}  \x_i}+\frac{\alpha}{1+\alpha}\bm{I}
                   \label{eq:iterations_reg_TME}
\eea
converge to the unique solution.
Various properties of TME and its regularized variant, in the  limit as \(p,n\to\infty\) with $p/n\to\gamma$, were proven by \citet{dumbgen_tylers,MP,MP1,MP2}.

The following lemma, proven in the appendix, shows that 
if $\alpha$ is sufficiently large and $\bm \Sigma(\alpha)$  exists, then the iterations (\ref{eq:iterations_reg_TME}), starting  from \(\bS_1(\alpha)={\alpha}\bI/{(1+\alpha)} \), have a uniform linear convergence rate already from the  first iteration.
To the best of our knowledge, this result is new and is of independent interest.

To state the lemma, let $e_k=\| \bSa-\bS_k(\alpha) \|$ be the error after $k$ iterations, $\bm \tilde{\bm X}$ be the $p \times n$ matrix whose columns are $\{\x_i/\|\x_i\|\}_{i=1}^{n}$ and let
$$C(\bm \tilde{\bm X})= \frac{p}{n} \| \bm \tilde{\bm X} \tilde{\bm X}^T\|  
=  \frac{p}{n} \left\| \sum_{i=1}^n \frac{\x_i \x_i^T}{\| \x_i \|^2} \right\|.$$
Note that for a given dataset, $C(\bm \tilde{\bm X})$ is fixed and can be computed a-priori.

\begin{lem} \label{lem:convergence_reg_TME}
Let $\{\bm{x}_i\}_{i=1}^n$ be a data set in $\mathbb{R}^p$ with constant $C(\bm \tilde{\bm X})$ and let $0<R<1$. Suppose that $\alpha>  \max((3+R^{-1}) C(\bm \tilde{\bm X})-1,0)$
and let $\hat{\bm\Sigma}(\alpha)$ be a solution of \eqref{regTME}. Then,
the iterations of Eq.~(\ref{eq:iterations_reg_TME}), starting from $\bS_1(\alpha)=\frac{\alpha}{1+\alpha}\bI$, {{uniformly and linearly converge}} to $\hat{\bm\Sigma}(\alpha)$ with the ratio $R$. That is,
\begin{equation}
e_{k+1}\leq R e_{k} \leq R^k e_1, \ \text{for all } k \geq 1.
        \label{eq:linear_convergence_TME}
\end{equation}
\end{lem}

A straightforward calculation yields the bound $C(\bm \tilde{\bm X}) \geq p/n$. Hence, the above assumptions on $\alpha$ imply that $\alpha>\max(0,p/n-1)$ and consequently guarantee the
existence and uniqueness of $\bSa$ in our setting.

Lemma \ref{lem:convergence_reg_TME} implies that calculating $\hat{\bm{\Sigma}}(\alpha)$ is computationally efficient, since for accuracy $\epsilon$ and convergence ratio $R$, at most $\lceil \log_{R^{-1}}(\epsilon^{-1}) \rceil$
iterations are needed. If \(n>p\) then at each iteration, the matrix inversion costs $O(p^3)$ operations and the other operations are $O(n \, p^2)$. For \(n<p\) one may first perform an SVD of the data and compute the subspace $W=\text{Span}(\x_i)$ whose dimension 
is at most $n$. Since for any $\bm v\bot W$, by definition $\hat{\bm\Sigma}(\alpha)\bm v = \alpha/(1+\alpha)\bm v$, it suffices to calculate $\hat{\bm \Sigma}(\alpha)$ restricted to the subspace $W$. Each iteration then costs at most $O(n^3)$. 
 Therefore, for sufficiently large $\alpha$, the total cost of computing $\hat{\bm{\Sigma}}(\alpha)$ within accuracy $\epsilon$ is $O(\log(\epsilon^{-1}) (n+p) \min(n,p)^2)$.

Our theoretical analysis below studies the regularized TME as \(p,n\to\infty\) and $p/n\to\gamma\in(0,\infty)$, but with a fixed value of \(\alpha\). The next lemma shows that for data sampled from an elliptical distribution, with high probability   $C(\bm \tilde{\bm X})$ is bounded by a constant that depends on $\|\Sp\|$ and on the ratio \(p/n\).

\begin{lem}\label{lem:CX_bound}
    Let $\bm x_1,\dots,\bm x_n$ be i.i.d. from Eq. (\ref{data}) with $\bm\mu = \bm 0$ and shape matrix $\Sp$. Then, with probability $>1-\exp(-cp)$, where  $c=c(\|\Sp\|)>0$, 
\begin{equation}
C(\bm \tilde{\bm X})\le
    2\|\Sp\|\Big(1+2\sqrt{p/n}\Big)^2.
        \label{eq:bound_CX}
  \end{equation}
\end{lem}


\subsection{TME-Based Thresholding Estimators}\label{Sec:Threshold_TME}

%
One possible approach to construct a sparse and robust estimator for the shape matrix is to add a suitable penalty to the original cost functional Eq. (\ref{eq:Loss_TME}) of the TME. For various structural assumptions on the shape matrix, this approach was proposed by \citet{tylercov}
and by \cite{sun2016robust}. 

For a sparsity inducing penalty, however,
such an approach would in general lead to a non-convex and potentially difficult to optimize objective. 
As such, we instead opt for thresholding the (regularized) TME, which as we show in our paper, for sufficiently large regularization $\alpha$, can be computed efficiently in practical polynomial time. 

For a matrix $\bm A=(a_{ij})$ and threshold $t>0$,  define the hard-thresholding
operator by
\bea
\tau_t(\bm A)
= ({\bm1}(|a_{ij}|>t)a_{ij}).
    \nonumber
\eea
For $n>p$, where the TME $\hat{\bm \Sigma} $ exists and by definition has unit
trace, our proposed estimator for the shape matrix \(\Sp\) takes the form
\be
\hat{\bm S}_p=\tau_t\left({\hat{\bm \Sigma}}\right),
        \label{eq:T_TME}\
\ee
where the threshold $t=t(p,n)$ is specified below.  Similarly, for general
$p,n$, 
our estimator based on the regularized TME is
\be
\hat{\bm S}_p = \tau_t\left(p\frac{
\hat{\bm \Sigma}(\alpha)-\tfrac{\alpha}{1+\alpha}\bm I}
{\tr(\hat{\bm \Sigma}(\alpha)-\tfrac{\alpha}{1+\alpha}\bm I)}\right).
        \label{eq:T_R_TME}
\ee
Note that both $\hat{\bm{ \Sigma}}$ in Eq. (\ref{eq:T_TME}) and the argument matrix prior to thresholding in Eq. (\ref{eq:T_R_TME}) have rank at most $\min(n,p)$. 

%
\subsection{Accuracy of the Thresholded TME} \label{sec:Accuracy_Thresholded_TME}

Theorems~\ref{plessn} and \ref{regthm}, proved in Sections \ref{unreg} and \ref{reg}, respectively, establish the asymptotic
accuracy of Eqs. (\ref{eq:T_TME}) and (\ref{eq:T_R_TME}) as estimates of
the shape matrix $\bm S_p$.


\begin{thm} \label{plessn}
Consider a sequence \((n,p,\Sp)\) where \(n\to \infty,$  $p=p_n\to\infty\) with $p/n\to\gamma\in(0,1)$, and
$\Sp \in\mathcal{U}(q,s_p,M)$. For each triplet $(n,p,\bm S_p)$,
let \(\hat{\bm \Sigma}\) be the TME of \(n\) i.i.d. samples $\{\bm x_i\}_{i=1}^n\subset \R^p$  from the elliptical distribution (\ref{data}). Then there
exists a constant $M'$ depending only on $\gamma$ such that for any fixed
$M''>M'$, the thresholded TME of Eq.~(\ref{eq:T_TME}) with threshold
$t_n=M''\sqrt{{\log p}/{n}}$,
approaches $\bm{S}_p$ in spectral norm at a rate
\begin{equation}
\Big\|\tau_{t_n}(p\bm{\hat\Sigma})-\bm{S}_p \Big\| = \mathcal{O}_P\bigg( s_p\cdot \left(\frac{\log p}{n}\right)^{(1-q)/2}\bigg).
                \nonumber
\end{equation}
\end{thm}

\begin{thm}\label{regthm}
Consider a sequence  \((n,p,\bm S_p)\) as in Theorem \ref{plessn},
here with \(p/n\to\gamma\in(0,\infty)\) and with the additional assumption
that $\|\Sp\| \leq s_{\max}$.
For ${\alpha}>\max(0,\gamma-1+s_{\max}(1+\sqrt{\gamma})^2)$, let \(\hat{\bm
\Sigma}(\alpha)\) be the regularized TME of \(n\) i.i.d.~samples  $\{\bm
x_i\}_{i=1}^n\subset\R^p$  from the elliptical distribution (\ref{data}).
Then there exists an $M'$ depending only on
$\gamma$ and $\alpha$ such that for any fixed $M''>M'$, the estimator of Eq.~(\ref{eq:T_R_TME}) with $t_n=M''\sqrt{{\log p}/{n}},$ converges in
spectral norm to $\bm S_p$ at rate
\bea
\left\| \tau_{t_n}\left(p\frac{\left(\hat{\bm \Sigma}(\alpha) - \frac{\alpha}{1+\alpha} \bm I
        \right)}{\tr\left(\hat{\bm \Sigma}(\alpha) - \frac{\alpha}{1+\alpha} \bm I
\right)} \right)-
\bm S_p  \right\| = \mathcal{O}_P\left(s_p \left(\frac{\log p}{n}
\right)^{(1-q)/2}\right). \nonumber
\eea
\end{thm}
Several remarks regarding Theorem \ref{regthm} are in place. 

\begin{rek}
As noted by \citet[p.~2580]{CT}, if $\Sp\in \mathcal U(q,s_{p},M)$ 
then $\|\Sp\|\leq M^{1-q} s_p$ which may grow with \(p\). Since we analyze the regularized TME with a {\em fixed} value of $\alpha$, we explicitly require that
$\|\Sp \|\leq s_{\max}$ independent of $p$. If the sequence of matrices $\Sp$ has a norm that grows to infinity with $p$, then the regularization $\alpha$ should also grow to infinity with $p$, such that $\alpha > c \|\Sp\|$ for some constant $c>0$. We believe an analogue of Theorem \ref{regthm} should hold in this case, but this requires a careful analysis beyond the scope of this paper. 
\end{rek}

\begin{rek}
The convergence rate in  Theorems~\ref{plessn} and \ref{regthm} coincides with the
minimax optimal rate for sparse covariance estimation with sub-Gaussian data,
derived by \citet{minimax}.
Since the Gaussian distribution is a particular case of an elliptical distribution, our estimators are thus minimax rate optimal. Furthermore, in light of Lemmas \ref{lem:convergence_reg_TME} and \ref{lem:CX_bound}, computing the regularized TME and subsequently thresholding it, is computationally efficient.
\end{rek}

\begin{rek}
One use of regularized variants of Tyler's M-estimator is to provide an accurate estimate of the shape matrix when $p>n$.
With this goal in mind, choosing the precise value of the regularization constant is crucial \citep{chen,MP1}.  Setting the regularization parameter
is also important to maximize the asymptotic detection probability in various signal processing applications \citep{kammoun2018optimal}. In contrast, 
in our case, as we remove the regularization $\alpha/(1+\alpha) \bm I$ prior to thresholding, at least asymptotically, 
the precise value of the regularization parameter is unimportant, provided it is sufficiently large. This is also evident in the simulations described in 
Section \ref{simulations}. From a practical perspective, we thus suggest to use a value of $\alpha$ as described in Lemma \ref{lem:convergence_reg_TME}, with say $R=1/2$, which is not only sufficient for existence but also guarantees fast convergence of the iterations to compute the regularized TME. 
\end{rek}

\section{Preliminaries}

In proving  Theorems \ref{plessn} and \ref{regthm}, we shall make frequent use of the following auxiliary lemmas. The first is a simple inequality. Let \(A,B\) be non-negative random variables. Then for any \(c>0\) and $\lambda >0$,
\begin{equation}
\Pr(AB>c)\leq \Pr(A>\lambda c) +\Pr(B>1/\lambda).
        \label{eq:ABc_lambda}
\end{equation}
Next, is the following well known result, which shows that TME and regularized TME are unable to
distinguish an elliptical distribution from a Gaussian one. Its proof (omitted)\ follows directly from the fact that (regularized) TME for data $\x_i$ is identical to that of data $t_i\x_i$, where $t_i$ are arbitrary positive real valued numbers.

\begin{lem}\label{lem:scaling}
TME or  regularized TME with $\alpha>\max(0,p/n-1)$
  under an elliptical
distribution with shape matrix $\bm S_p$ has the same distribution as under a Gaussian distribution with covariance  $\bm S_p$.
\end{lem}

The following two results from random matrix theory will also be of use. The first is a non-asymptotic bound on the spectral norm of a Wishart matrix, and the second on the concentration of quadratic forms. See for example \citep{davidson}[Theorem 2.13]  and \citep{rudelson2013}[Theorem 1.1].

\begin{lem}  \label{lem:davidson_bound}
    Let $\{\bm \xi_i\}_{i=1}^n \subset \mathbb{R}^p$ be
    i.i.d. $N(0,\bm I)$, and let $\bm T_n=\frac{1}n\sum_i \bm \xi_i\bm \xi_i^T$. Then, $\E[\|\bm T_n\|]\leq (1+\sqrt{p/n})^2$ and 
\bea
\Pr\left(  \| \bm T_n\| >\bigg(1+\sqrt{p/n}+t\bigg)^2   \right) \le
\exp\left(-nt^2/2\right).
    \nonumber
\eea
\end{lem}

\begin{lem}\label{Rudelson}
Let $\bm A \in \mathbb{R}^{p\times p}$ and $\bm \xi\sim N(\bm 0,\bm I)$. Then, there exist absolute constants $c_{1},c_{2}>0$ such that for all $\epsilon>0$,
\bea
\Pr\left(\left|\bm \xi^T \bm A \bm \xi - {\tr}\left(\bm
A\right)\right|> \epsilon \right)\le 2
\exp\left(-c_1\min\left\{\frac{c_2^2\epsilon^2}{\left\|\bm A\right\|_F^2},\frac{c_2\epsilon}{\left\|\bm A\right\|}  \right\}\right).
    \nonumber
\eea
\end{lem}

Finally, the following auxiliary lemma, proved in Appendix \ref{sec:proof_b1}, is a slight modification of a result by \citet[p.~2583]{CT}. 
\begin{lem} \label{bl}
Assume $\bm{B}\in \mathcal{U}(q,s_p,M).$ Let $\bm{A}$ be a matrix such that \bea
\big\|\bm{A}-\bm{B}\big\|_{\max} \le C_1 \sqrt{\log p/n},
    \nonumber
\eea
for some $C_1>0$. Suppose we threshold \(\bm A\) at level $t=K\sqrt{\log p/n}$, with $K>C_1$. Then, there exists a constant $C_2=C_2(C_1,K,q)<\infty$ such that
\bea
\big\|\tau_{t}(\bm{A}) - \bm{B}\big\| \le C_2 s_p(\log p/n)^{(1-q)/2}.
    \nonumber
\eea
\end{lem}

\section{Proof of Theorem~\ref{plessn}} \label{unreg}

The proof consists of three main steps: (i) reducing to a bound on
$\|\bm{\hat\Sigma}-\bm{\hat S}\|_{\max}$; (ii) expressing $\hat{\bm\Sigma}$ as a weighted covariance matrix whose coefficients are all uniformly
close to a constant, with high probability; and (iii) bounding $\|\bm{\hat\Sigma}-\bm{\hat S}\|_{\max}$.

\subsection{Step 1: Reduction from $\|\tau_{t_n}(\bm{\hat\Sigma})-\bm{S}_p \|$ to  $\|\bm{\hat\Sigma}-\bm{\hat S}\|_{\max}$}\label{BLlemma}

By Lemma \ref{bl}, it suffices to prove that
$\|\bm{\hat\Sigma}-\bm{ S}_p\|_{\max}=\mathcal{O}_P(\sqrt{\log p /n})$.
 Let $\hat{\bm S}$ be the sample covariance of  $\{\bm x_i\}_{i=1}^n$. By the triangle inequality,
\bea
\|\hat{\bm{\Sigma}}-\bm{S}_p\|_{\max}\le \|\hat{\bm{\Sigma}}-\bm{\hat
S}\|_{\max} + \|\bm{\hat S}-\bm{S}_p\|_{\max}.
    \nonumber \label{tri}
\eea
In light of Lemma \ref{lem:scaling}, we may assume that  $\x_i$ are all i.i.d. $N(\bm 0,\bm S_p)$ .  Since the proof of Theorem 1 of \citet{CT} shows that
\bea
\|\bm{\hat S}-\bm{S}_p\|_{\max}=\mathcal{O}_P\left(\sqrt{\log p/n}\right)
        \nonumber
\eea
it thus suffices to show that
\bea
\|\bm{\hat\Sigma}-\bm{\hat S}\|_{\max}=\mathcal{O}_P\left(\sqrt{\log p /n}\right).
\label{eq:specbdd}
\eea

\subsection{Step 2:\ The weights of TME} \label{Relevant}

By \citet[Lemma 2.1]{MP}, TME has an equivalent definition  as a weighted covariance matrix,
\bea
\bm{\hat\Sigma} = p  \sum_{i=1}^n  w_i \bm{x}_i \bm{x}_i^T \Big/ \tr\Big(\sum_{i=1}^n  w_i \bm{x}_i \bm{x}_i^T \Big),
    \nonumber
\eea
where the weights $w_i$ are the unique solution of
\bea
        \underset{{w_i>0, \sum w_i = 1}}{\arg\min}
                -\sum_{i=1}^n \ln w_i + \frac np \ln\det
\bigg(\sum_{i=1}^n w_i\bm x_i \bm x_i^T \bigg).
                \label{uniquewts}
\eea
This characterization is important because of the following result:

\begin{lem} \label{lem:weights}
Consider a sequence $(n,p,\Sp)$ where \(n,p\to\infty\) with \(p/n\to\gamma\in(0,1)\),
and  \(\Sp \in S_{++}^p\). For every triplet \((n,p,\Sp)\), let $\x_i\overset{iid}\sim N(\bm{0},\bm S_p)$ and let
$\{ w_i\}_{i=1}^n$ be the corresponding weights of Eq.~\eqref{uniquewts}.
Then there exist positive constants $C,c$ and \(c'\) depending only on $\gamma$ such that for any $0<\epsilon<c'$,
and sufficiently large \(n\),
\bea
{\rm {Pr}}\left[\max_i |n w_i-1|\ge\epsilon \right]\le Cne^{-c\epsilon^2 n}.
    \label{eq:weightseq}
\eea
\end{lem}
The  case $\Sp=\bm{I}$ was proved by~\citet[Lemma 2.2]{MP}. Its generalization to an arbitrary $\Sp \in S_{++}^p$ is proved in Appendix~\ref{weightspf}.

\subsection{Step 3: Bounding $\| \hat{\bm \Sigma} -\hat{\bm S}\|_{\max}$}
                \label{fourpt3}

The proof of Theorem \ref{plessn} concludes by applying the following lemma
which establishes Eq.~\eqref{eq:specbdd}. Its proof is in Appendix~\ref{pfspec}.

\begin{lem}\label{spec} Let $\hat{\bm \Sigma}$ and \(\hat{\bm S}\) be the TME and
    the sample covariance matrix of $\bm{x}_1,\dots,\bm{x}_n$  i.i.d.~from
    $N(\bm0, \Sp)$, where $\Sp \in \mathcal U(q,s_{p},M)$ with \(tr(\Sp)=p\). Assume that $p,n\to\infty$, with
    $p/n\to\gamma\in(0,1)$. Then there exist positive constants $C,c$ and
    \(c'\) that depend only on $\gamma$, such that for all $\epsilon\in(0,c')$
and $n$ sufficiently large
\bea
{\rm {Pr}}\left(\| \bm{\hat\Sigma} - \bm{\hat S}\|_{\max} \ge \epsilon\right)\le Cne^{-c\epsilon^2 n}.
    \nonumber \label{spectral}
\eea
\end{lem}

\section{Proof of Theorem~\ref{regthm}} \label{reg}

We first introduce and prove a slightly modified version of Theorem \ref{regthm}. We then show how
Theorem~\ref{regthm} follows from it.
The modified theorem  uses the following proposition, proved in Appendix~\ref{r_exists_pf}.

\begin{prop}\label{prop:r_finite}
Let $\bm y,\bm \xi_1,\dots,\bm
    \xi_{n-1}\in \R^p$ be i.i.d. $N(\bm 0,\bm I)$ and denote
    \bea
    Q=Q(r)=\frac 1p \bm y^T\left(\frac1n\sum_{j=1}^{n-1}\bm \xi_j \bm \xi_j^T + \alpha\frac np\frac
    1r \bm S_{p}^{-1} \right)^{-1} \bm y,\nonumber
    \eea
    where $\Sp\in \mathcal U(q,s_p,M)$ with $\|\Sp\|\leq s_{\max}$. 
    Assume that $\alpha>\max(0,p/n-1+s_{\max}(1+\sqrt{p/n})^2)$, and define \[
r_{\min}=\frac{n}{p}\frac{\alpha }{1+\alpha-p/n},\quad 
r_{\max}=\frac{n}{p}\frac{\alpha }{1+\alpha-p/n-s_{\max}(1+\sqrt{p/n})^2} . 
\] 
Then, there exists a unique $r=r(p,n,\alpha,\Sp)\in[r_{\min},r_{\max}]$, such that
 \begin{equation}
 \mathbb E[Q(r)] =
\frac{1}{1+\alpha-p/n}, \label{eq:EQ}
 \end{equation}
 where the expectation is over $\bm y$ and $\bm \xi_1,\ldots,\bm \xi_{n-1}$.
 \end{prop}

\subsection{A reformulation of the main result}\label{overview_pf}

We now introduce the modified theorem.
\begin{thm}\label{regthm2}
Consider the same setting as in Theorem \ref{regthm}. Then there exists an $M'$ depending only on $\gamma$ and $\alpha$ such that for any fixed $M''>M'$,
the estimator $\tau_{t_n} (\hat{\bm \Sigma}(\alpha) - {\alpha}\bm I/(1+\alpha))$ with
$t_n=M''\sqrt{\frac{\log p}{n}},$ converges in spectral norm to a multiple of $\bm S_p$,
\bea
\left\| \tau_{t_n}\left(\hat{\bm \Sigma}(\alpha) - \frac{\alpha}{1+\alpha} \bm I \right) -
\frac pn\frac r{1+\alpha}\bm S_p  \right\| = \mathcal{O}_P\left(s_p
\left(\frac{\log p}{n} \right)^{(1-q)/2}\right), \nonumber
\eea
where the scalar $r=r(p,n,\alpha,\Sp)$ is specified in
Proposition~\ref{prop:r_finite}.
\end{thm}


\subsection{Proof of Theorem~\ref{regthm2}}
\label{overview_pf2}

By Lemma~\ref{lem:scaling}, we may assume
$\bm x_i\overset{iid}\sim N(\bm0,\bm S_p)$. Following the argument in Section~\ref{BLlemma}, 
combining Lemma \ref{bl} with the fact that by Proposition \ref{prop:r_finite}, $r<r_{\max}$, it suffices to show that
\bea
\left\|\left(\hat{\bm{\Sigma}}(\alpha)-\frac{\alpha}{1+\alpha} \bm{I}\right) -
\frac pn\frac r{1+\alpha}\bm{\hat S}\right\|_{\max} =
\mathcal{O}_P\left(\sqrt{\log p  / n}\right).
    \label{eq:star}
\eea

Our proof proceeds as follows: First, we express
$\hat{\bm\Sigma}(\alpha)$ as the sum of $\frac{\alpha}{1+\alpha}{\bm I}$
and weighted $\bm x_i \bm x_i^T$ terms, where the weights
are the root of some equation. Next, we show that this
root is concentrated near the vector $ r{\bm1}/n$, with $r$ specified in Proposition~\ref{prop:r_finite}. Finally, we establish Eq.~\eqref{eq:star}.

Following the definition of the regularized TME, we write
$\hat{\bm \Sigma}(\alpha
)$ as
\bea
\hat{\bm \Sigma}(\alpha)= \frac{1}{1+\alpha}\frac pn\sum_{i=1}^n w_i \bm
x_i\bm x_i^T + \frac{\alpha}{1+\alpha}\bm I,  \label{tw9}
\eea
where
the weight vector $\bm w=(w_1,\dots,w_n)^T$ satisfies
\bea
w_i = \frac{1}{\bm x_i^T \hat{\bm \Sigma}(\alpha)^{-1}\bm x_i} = \frac{1}{\bm x_i^T \left(\frac1{1+\alpha}\frac pn\sum_{j=1}^n w_j\bm x_j
        \bm x_j^T  + \frac{\alpha}{1+\alpha}\bm I\right)^{-1}\bm x_i}.
        \label{weightsinw}
\eea
By \citet{sun}[Theorem 11],
$\hat{\bm\Sigma}(\alpha)$ is unique.

Next, consider the function
$g:\R^n\to\R^n$ whose \(n\) components  are
\bea
g(\bm v)_i = v_i-\frac{1}{\bm x_i^T\left(\frac{1}{1+\alpha}\frac pn
\sum_{k=1}^n v_k \bm x_k\bm x_k^T + \frac{\alpha}{1+\alpha}n \bm
I\right)^{-1}\bm x_i}.
    \label{eq:gee}
\eea
Comparing Eq.~\eqref{eq:gee} to Eq.~\eqref{weightsinw}, the  \(n\) non-linear equations $g(\bm v)=\bm 0$ have a unique solution, which is thus $n\bm w$.
The next three lemmas state  properties of \(g\) used to prove that as \(p,n\to\infty\), with $p/n\to\gamma$, this root
concentrates around
$\bm u = r\bm1$, with \(r\) given in Proposition~\ref{prop:r_finite}.
The lemmas, proven in Appendices~\ref{pflem1}--\ref{pflem3}, assume the setting of Theorem~\ref{regthm2}, and their generic constants depend only on \(\gamma,\alpha\) and \(s_{\max}\). Our analysis of the weights \(w_i\) follows the pioneering works of \citet{couillet2014robust,couillet2015random}, who proved that
the weights in Maronna's M-estimators converge to suitable constants, and \cite{MP}, who derived concentration results for the weights of Tyler's M-estimator as $p,n\to\infty$ with $p/n\to\gamma < 1$. 

\begin{lem}\label{lem1}
There exist $C,c>0$ such that for any $\epsilon\in(0,1)$
\bea
{\rm{Pr}}\left(\left\| g(\bm u) \right\|_\infty > \epsilon
\right)< Cpe^{-cp\epsilon^2}.
    \nonumber \label{eq:lem1}
\eea
\end{lem}

\begin{lem}\label{lem2}
There exist  $c',c_{L},C,c>0$ such that
\bea
\Pr \left(\exists{\bm v}\in B_{c'}(\bm u), \left\|\nabla g\left(\bm v\right) -\nabla g\left(\bm
u\right) \right\|_{\max}> c_{L}\|\bm v - \bm u\|_{\infty} \right) < C p^{2} e^{-cp}.
    \nonumber
\eea
\end{lem}

\begin{lem}\label{lem3}
There exist $c_{H},C,c>0$ such that
\bea
{\rm{Pr}}\left(\left\|  \left( \nabla g\left(\bm u\right) \right)^{-1}
\right\|_\infty > c_{H}  \right) < Cpe^{-cp}.
    \label{eq:lem3}
\eea
\end{lem}

Lemmas~\ref{lem1} and~\ref{lem2} show that w.h.p.~$g(\bm u)$ is small and  $\nabla g$ is Lipschitz near $\bm u$. These two properties are consistent with the root of \(g\) being close to $\bm u$.
To rigorously prove this, following \cite{MP}, we consider the function
$f(\bm v) = \left(\nabla g(\bm
u)\right)^{-1}g(\bm v)$. Lemma~\ref{lem3} shows that the matrix  $(\nabla g(\bm u))^{-1}$ is
w.h.p.~not extremely large. Finally, the following lemma
combines these properties of $g$ to infer that its root is close to \(\bm u\).
\begin{lem}\label{Teng_cor}
Let $f:\R^n\to\R^n$, $\bm u\in\R^n$  and $C>0$. Assume that
\begin{enumerate}
    \item $\nabla f(\bm u)=\bm I$;
    \item
        $\|\nabla f(\bm v) - \nabla f(\bm u) \|_{\max}  \leq C\|\bm v-\bm u\|_\infty$ for all
        $\|\bm v-\bm u\|_\infty \le 3\|f(\bm u)\|_\infty$;
    \item $\|f(\bm u)\|_\infty<\min\{1/(9C),1/3\}$.
\end{enumerate}
Then there exists a   $\bm{\tilde v}\in\mathbb{R}^n$ such that $f(\bm{\tilde v})=\bm 0$ and $\|\bm{\tilde v} - \bm u\|_\infty < 3\|f(\bm u)\|_\infty$.
\end{lem}
Lemma~\ref{Teng_cor} is slightly stronger than  Lemma 3.1 of~\citet{MP}, as it has a weaker  requirement that the Lipschitz condition in Lemma \ref{Teng_cor} holds in a smaller ball $\|\bm v-\bm
u\|_\infty \le 3\|f(\bm u)\|_\infty$, instead of the original requirement $\|\bm v-\bm u\|_\infty<1$ in their Lemma 3.1. A careful inspection shows that their original proof is still valid under this weaker assumption.

To apply Lemma~\ref{Teng_cor} to $f(\bm v) = \left(\nabla g(\bm
u)\right)^{-1}g(\bm v)$, we verify that the three conditions of the lemma  hold with high probability. The
first condition is trivially satisfied. For the other two conditions,
by Lemmas \ref{lem1} and
\ref{lem3},  w.h.p.
\[
\|f(\bm u)\|_\infty \leq \|(\nabla g(\bm u))^{-1}\|_\infty \cdot \|g(\bm u)\|_\infty \leq c_H\epsilon.
\]
Similarly,
by Lemmas~\ref{lem2} and~\ref{lem3},
for all \(\|\bm v-\bm u\|_\infty\leq c'\), w.h.p.
\[
\|\nabla f(\bm v)-\nabla f(\bm u)\|_{\max}
\leq
\|(\nabla g(\bm u))^{-1}\|_\infty \cdot\|\,\nabla g(\bm v)-\nabla g(\bm
u) \|_{\max}
\leq
c_H  c_L \|\bm v-\bm u\|_\infty.
\]
Since for sufficiently small
$\epsilon$, $c_{H}\epsilon<\min\{1/(9c_L c_H),1/3\}$, both the second and third
conditions of Lemma~\ref{Teng_cor} are thus satisfied with constant \(C=c_{L}c_H\).

To conclude, with probability at least $1-Cp^{2}e^{-cp\epsilon^2}$all three conditions
of Lemma~\ref{Teng_cor} hold, so there exists $\tilde{\bm v}\in\R^n$ such that $f(\tilde{\bm
v})=\bm0$ and $\|\tilde{\bm v} - \bm u\|_\infty \le 3\|f(\bm
u)\|_\infty<3c_{H}\epsilon$. Since
$n\bm w$ is the unique root of $g(\bm v)$ and  also of $f(\bm
v)$,
\bea
\Pr\left( \|n\bm w - r\bm1\|_{\infty} > 3c_{H}\epsilon  \right) <
Cp^{2}e^{-cp\epsilon^2}.
    \label{eq:nwr1}
\eea

Next, we use Eq.~(\ref{eq:nwr1}) to bound the LHS of Eq.~\eqref{eq:star}. First,
 by Eq.~\eqref{tw9},
\bea
\left\|\left(\hat{\bm \Sigma}(\alpha) - \tfrac{\alpha}{1+\alpha}\bm I\right) -
\tfrac{1}{1+\alpha}\tfrac pn r \hat{\bm S}\right\|_{\max} &=& 
        \tfrac1{1+\alpha}\tfrac pn\left\|\sum_{i=1}^n w_i \bm
x_i\bm x_i^T -  r\tfrac1n\sum_{i=1}^n  \bm x_i\bm x_i^T \right\|_{\max}\nonumber \\
&\le& \tfrac{1}{1+\alpha}\tfrac pn  \left\|n\bm w - r\bm 1\right\|_\infty \left\|
\tfrac1n\sum_{i=1}^n \bm x_i \bm x_i^T
\right\|_{\max}.
    \nonumber
\eea

Using this inequality and Eq.~(\ref{eq:ABc_lambda}) with $\lambda =1/[s_{\max}(1+2 \sqrt \gamma)^{2}]$,
\begin{multline}
\Pr\left( \left\|(\hat{\bm \Sigma}(\alpha) - \tfrac{\alpha}{1+\alpha}\bm I) -
\tfrac{1}{1+\alpha}\tfrac pn r \hat{\bm S}\right\|  > \epsilon
\right)
\le \Pr\left(   \tfrac{1}{1+\alpha}\tfrac pn  \|n\bm w - r\bm 1\|_\infty \|
\hat{\bm S}  \| > \epsilon \right)
    \nonumber \\
\le \Pr\left(     \|n\bm w - r\bm 1\|_\infty  >
\epsilon \frac{n(1+\alpha)}{p s_{\max}(1+2\sqrt\gamma)^2} \right) \nonumber
+ \Pr\left(
\|\hat{\bm S} \| >
s_{\max}(1+2\sqrt\gamma)^2  \right).
    \nonumber
\end{multline}
Since \(\hat{\bm S}=\Sp^{1/2}(\frac1n\sum_i \bm\xi_i\bm\xi_i^T)\Sp^{1/2}\) with $\bm\xi_i\sim N(0,\bI)$, by Lemma~\ref{lem:davidson_bound} the second term is exponentially small in \(p.\)
By Eq.~\eqref{eq:nwr1}, the first term is bounded by \(C'p^{2}e^{-c p\epsilon^2}\). Hence, Eq.~\eqref{eq:star} holds, which concludes the proof of Theorem \ref{regthm2}.


\subsection{Concluding the proof of Theorem~\ref{regthm}} \label{3from2}

Similar to Theorems \ref{plessn} and \ref{regthm2}, to prove Theorem~\ref{regthm} it
suffices to show that
\be
\left\|p(\hat{\bm{\Sigma}}(\alpha)-\tfrac{\alpha}{1+\alpha}
\bm{I})/\tr(\hat{\bm{\Sigma}}(\alpha)-\tfrac{\alpha}{1+\alpha} \bm{I}) -\hat{\bm S}\right\|_{\max} =
\mathcal{O}_P\left(\sqrt{\log p  / n}\right).
    \label{eq:p_over_trace_2bdd}
\ee
Eq.~\eqref{eq:star} combined with Proposition~\ref{prop:r_finite}
imply that for \(r>r_{\min}>0\)
\be
\left\| \frac np\frac {1+\alpha}r \left(\hat{\bm{\Sigma}}(\alpha)-\frac{\alpha}{1+\alpha} \bm{I}\right) -
\hat{\bm S}\right\|_{\max} =
\mathcal{O}_P\left(\sqrt{\log p  / n}\right).
    \label{star_modified}
\ee
Since $\tr(\hat {\bm S})$ is tightly concentrated around $p$, we may replace $\hat{\bm S}$ in 
Eq. (\ref{star_modified}) by $p\hat{\bm S}/\tr(\hat{\bm S})$.
Now, 
Eq.~\eqref{eq:p_over_trace_2bdd} follows by the following lemma, proven in the appendix, combined
with the fact that w.h.p. $\|\hat{\bm S}\|_{\max} \leq 2 \|\Sp\|_{\max} \leq 2  M$.

\begin{lem} \label{lem:pA_trA} Let \(\bm B  \in S_{+}^p\)
with $\tr(\bm B)=p$ and $\|\bm B\|_{\max}\leq b_{\max}$. Suppose that  \(\bm A \in S_{+}^p \) satisfies $\|\bm A-\bm B\|_{\max}<\epsilon\leq 1/2$. Then,
\begin{equation}
\left\|\frac{p\bm A}{\tr(\bm A)}-\bm B\right\|_{\max}\leq 2(1+b_{\max})\epsilon.
\end{equation}
\end{lem}

\section{Numerical Experiments}\label{simulations}

Focusing on the regularized TME, we present simulations that support our theoretical analysis. Section \ref{num_subs1} compares the regularized TME, the sample covariance and their thresholded versions.
Section \ref{subs3} considers the sensitivity of the proposed estimator to $\alpha$.
Section \ref{sec:TME_outliers} demonstrates a simple modification of our estimator in the presence of outliers.

\subsection{Comparison of thresholded TME with covariance estimators} \label{num_subs1}

We considered the following shape matrix, also used by \cite{CT}:
\bea
\Sp = (s_{ij}) = (.7^{|i-j|}).
    \nonumber
\eea
Note that excluding the diagonal all rows of this matrix
have $\ell_1$ norm bounded by $2/(1-0.7)-2=14/3$. Hence, by the Gershgorin disk theorem, 
for any $p$, this matrix has 
a finite spectral norm, $\|\Sp\|\leq s_{\max}=1 + 14/3 = 17/3$. This is in accordance with
our assumptions in Theorem \ref{regthm}.

We generated data from a Gaussian scale mixture, which is a particular case of Eq. (\ref{data}). Here  $u$ and
$\bm\xi$  are independent, with $\bm \xi\sim N(\bm 0,\bm I)$. We considered three different choices for the random variables \(u_{i}\):\ (i)  $u_i= 1$, so
$\{\bm x_i\}_{i=1}^n$ are i.i.d.~$N(\bm 0,\bm S_p)$; (ii) $u_i \sim Laplace(0,1),$ a heavy tailed distribution with finite moments; and
(iii) $u_i\sim Cauchy(0, 1)$, so the distribution does not even have a well-defined mean or covariance.

We computed four estimators for the shape matrix: (i) SampCov: the sample covariance scaled to have trace \(p\),
$p\hat{\bm S}/\tr(\hat{\bm S})$; (ii) th-SampCov: the thresholded version of SampCov,
$\tau_t ( p \hat{\bm S} /\tr(\hat{\bm S}))$;
(iii) RegTME: the regularized TME, normalized to have trace \(p\),
\[
\frac{p({\bm \Sigma}(\alpha)-\tfrac{\alpha}{1+\alpha}\bm I)}
{\tr\left({\bm \Sigma}(\alpha)-\tfrac{\alpha}{1+\alpha}\bm I\right)
};
\]
and (iv) th-RegTME: the thresholded version of RegTME in  Eq.~(\ref{eq:T_R_TME}).
We choose 
$\alpha = 10$, and threshold at level $t=\sqrt{(\log p)/n}.$ Our stopping rule for  (\ref{eq:iterations_reg_TME})
is $\|p\hat{\bm \Sigma}_{k+1}/\tr(\hat{\bm \Sigma}_{k+1}) - p\hat{\bm \Sigma}_{k}/\tr(\hat{\bm\Sigma}_{k})\|_{F}<10^{-12}$,
or $k=1400$ iterations.

We measured the accuracy of an estimator $\hat{\Sp}$
by the logarithm of its averaged relative error (abbreviated LRE). That is, for 100 different realizations, we independently generated $n$ i.i.d.~samples in $\mathbb{R}^p$, and
each time estimated  $(\hat{\Sp})_{i}$, where $i=1,\ldots,100$. The LRE was then computed as follows
\bea
\operatorname{LRE} = \log\left( \frac1{100}\sum_{i=1}^{100}\frac{\| (\hat{\Sp})_{i}- {\bm S_p}  \|}
 {\|\Sp\|}\right).
    \nonumber
\eea
We considered sample sizes $n\in[100,1000]$ and the following  three ratios  $p/n\in\{.5,1,2\}$. Fig.~\ref{fig:fig1} shows the LRE of the four estimators.
As expected theoretically, for $u_i\equiv 1$   thresholding the sample
covariance or the regularized TME yield similar errors.
In contrast, for heavy-tailed data the thresholded sample covariance performs
poorly, whereas the thresholded
regularized TME is still an accurate estimate of \(\Sp\). Note
that since the regularized TME is invariant to the scaling $u_i$, the resulting errors of the regularized TME and its thresholded version (the blue squares and triangles) are the same for all three distributions of \(u_{i}\).

\begin{figure}[t!]
    \includegraphics[width=1.1\textwidth]{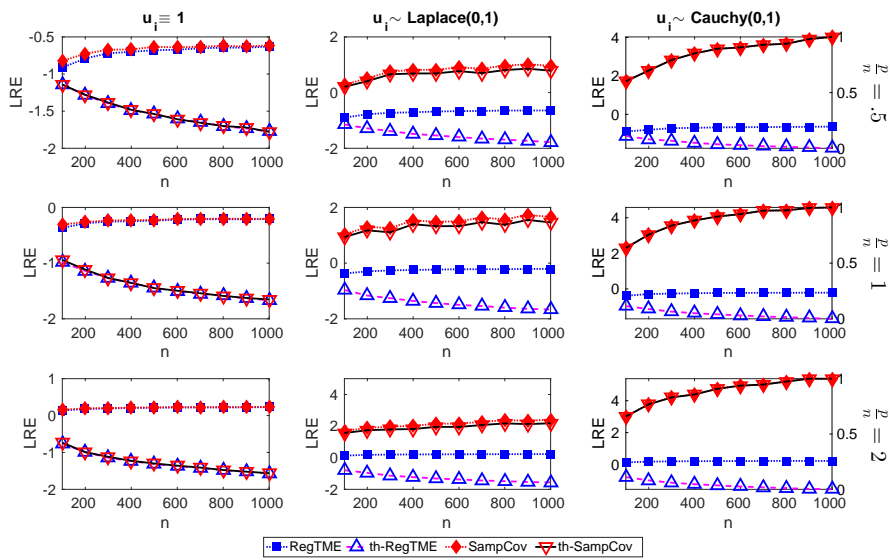}
    \caption{Comparison of the LRE of the four estimators
        with data  i.i.d.~from a Gaussian scale mixture. The rows correspond to \(p/n = 0.5,1,2\).
    The columns correspond to $u_i\equiv 1$, $u_i\overset{iid}\sim Laplace(0,1)$ and $u_i\overset{iid}\sim Cauchy(0, 1)$.
}
\label{fig:fig1}
\end{figure}

\subsection{Sensitivity of Regularized TME to Choice of $\alpha$}\label{subs3}

Next, we study how the error and runtime of th-RegTME depend on the
 regularization parameter
\(\alpha\).
We
consider the  Gaussian model with covariance $\bm S_p$,
and explore the behavior of th-RegTME for the following values of $\alpha$: 0.2, 0.4, 0.6, 0.8, 1, 2, $3 \ldots, 20$ and the following three cases:
$(p,n)=(800,400)$, $(p,n)=(800,200)$ and $(p,n)=(400,200)$.
Even though the regularized TME does not exist if  $\alpha <\max(0,p/n-1)$, our algorithm, with a stopping criterion based on scaled matrices converged for all considered values of $\alpha$. For a similar property upon scaling scatter matrices, see \cite{chen}.
 The left panel of Fig.~\ref{fig:fig5} shows the LRE of th-RegTME as a function of $\alpha$.
The maximal LRE occurs at $p/n-1$ and larger values of $\alpha$ yield slightly smaller
errors, which are nearly identical for all large values of $\alpha$.
This is in accordance with Theorem \ref{regthm}, which states that asymptotically all large values of $\alpha$ yield the same error rate. The right panel of Fig.~\ref{fig:fig5} displays the logarithm of the runtime of th-RegTME as a function of $\alpha$, showing a sharp increase in runtime as $p/n-1$
approaches $\alpha$.

\begin{figure}[t!]
\includegraphics[width=0.48\textwidth]{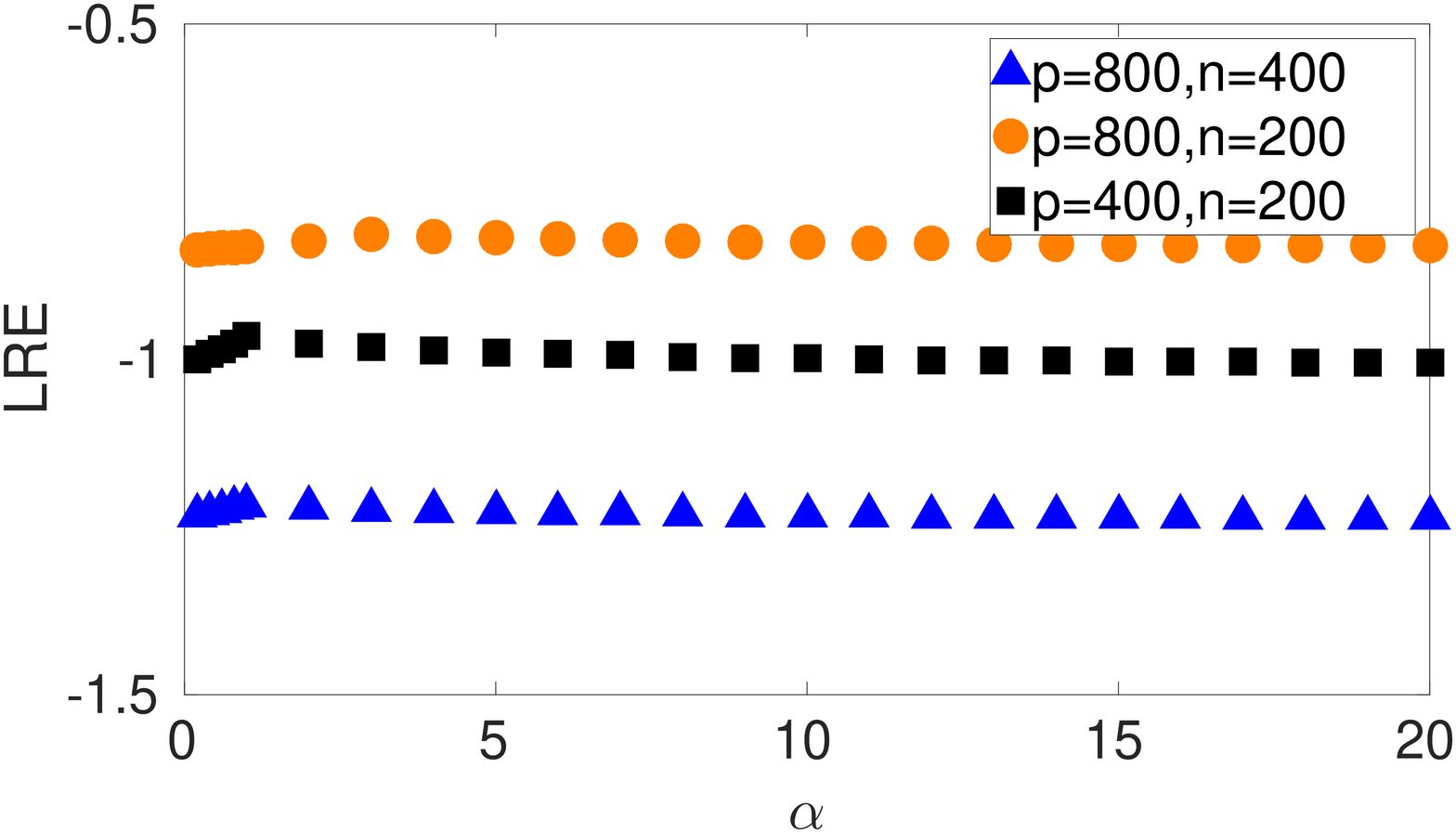}\,
\includegraphics[width=0.48\textwidth]{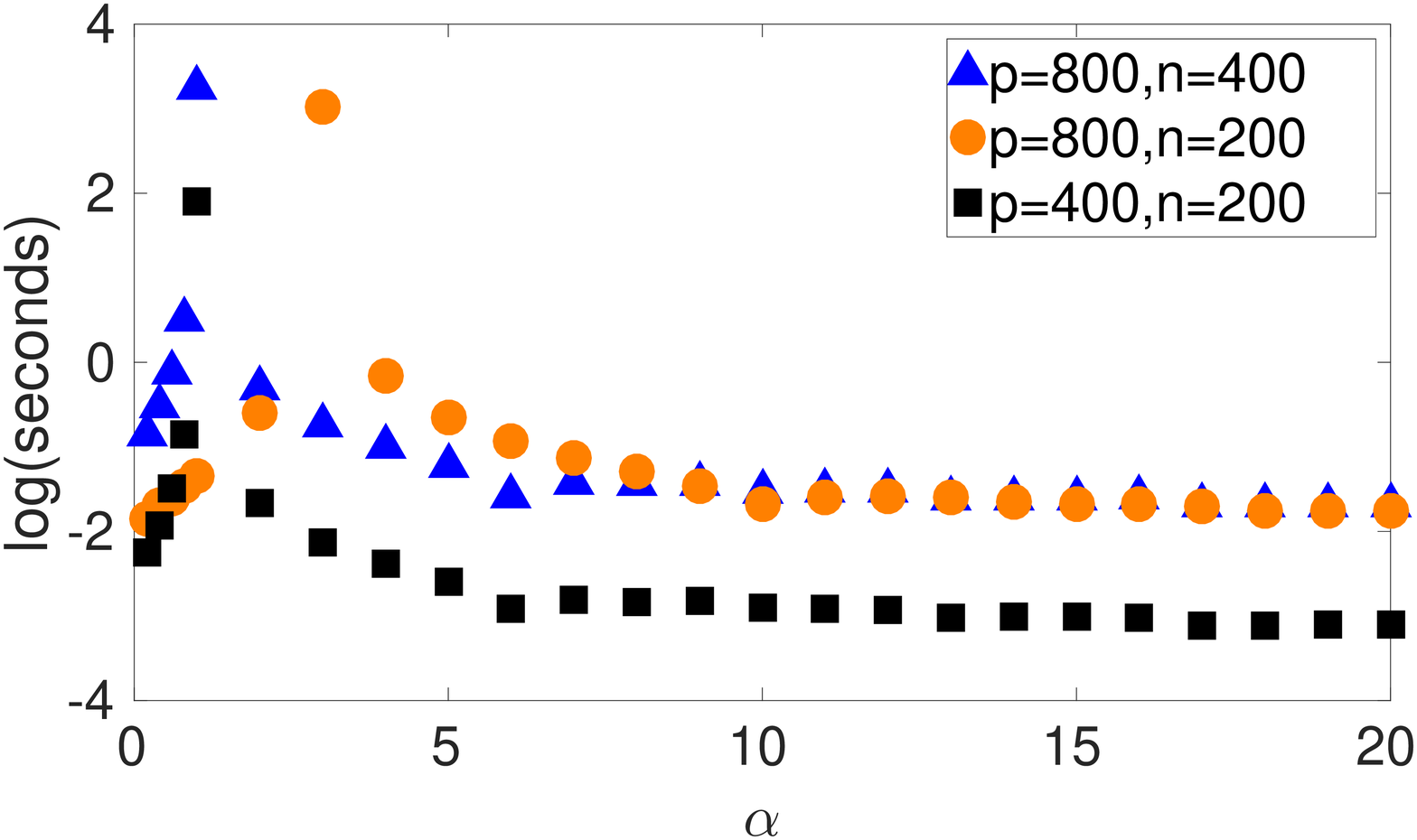}
\caption{LRE and log-runtime of th-RegTME on elliptical data for different choices of $\alpha$ and three choices of $p$ and $n$.}
\label{fig:fig5}
\end{figure}

Next, we explore the behavior of th-RegTME for $p=480$, $\alpha = 1, 2, 3, 4$ and $n = 60, 64, 68, \ldots, 300$.
The left panel of Fig.~\ref{fig:TME_alpha} shows the error of th-RegTME as a function of $n$. Again, in accordance with theory, \(\alpha\) has little effect on the accuracy.
Of particular interest is the runtime, seen in the right panel of
Fig.~\ref{fig:TME_alpha}. Here we see a sharp increase in runtime as $p/n-1$
approaches $\alpha$. For $n \geq \frac p{\alpha+1}$,  the runtime decreases as $\alpha$ increases.

These experiments indicate that one may generally prefer larger $\alpha$,
particularly for faster runtime.  We propose to choose a value of $\alpha$ close to the bound in Lemma \ref{lem:convergence_reg_TME} for some $R\in (0,1)$, which 
guarantees fast convergence.

\begin{figure}[t!]
\includegraphics[width=0.48\textwidth]{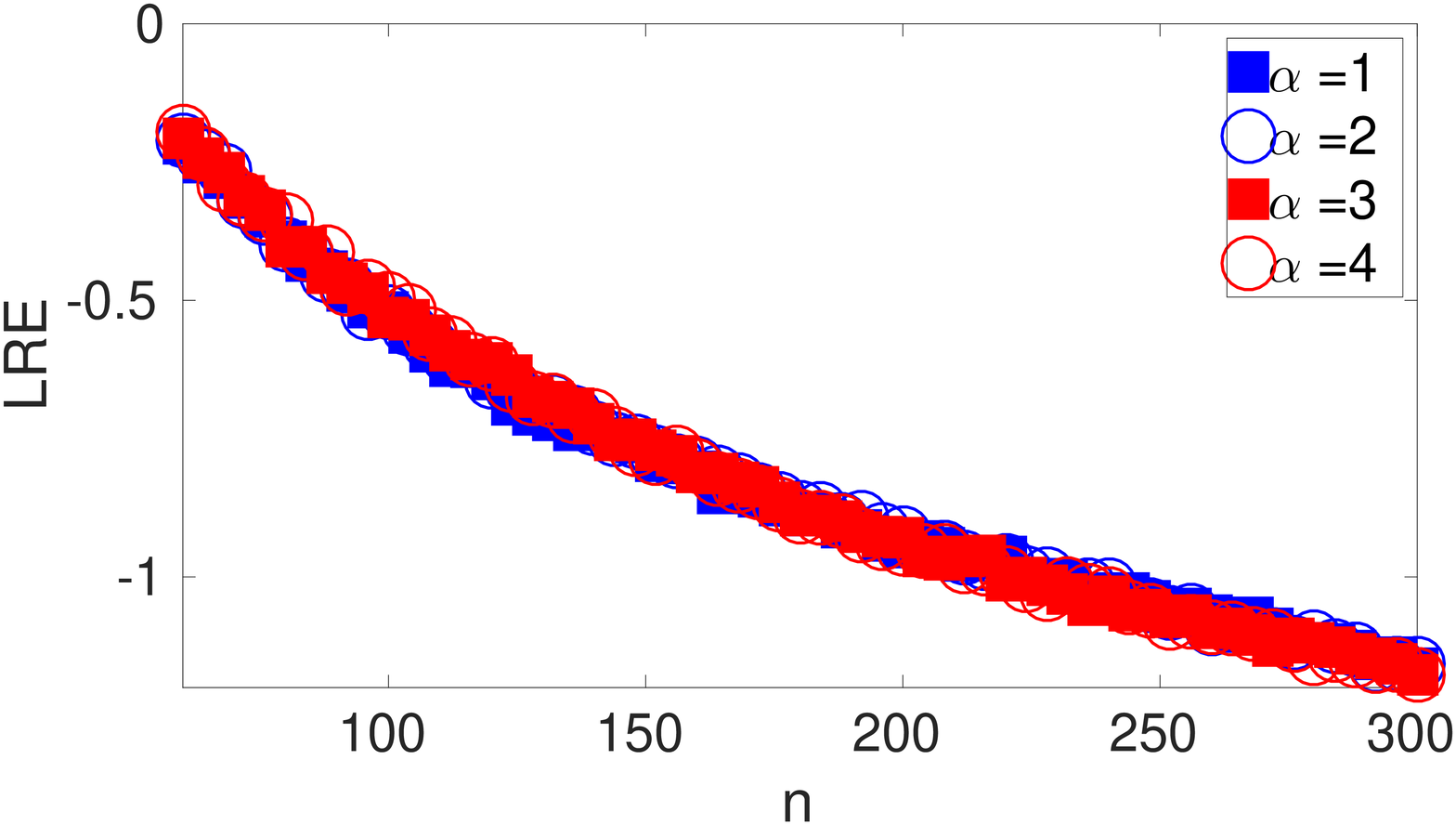}\,
\includegraphics[width=0.48\textwidth]{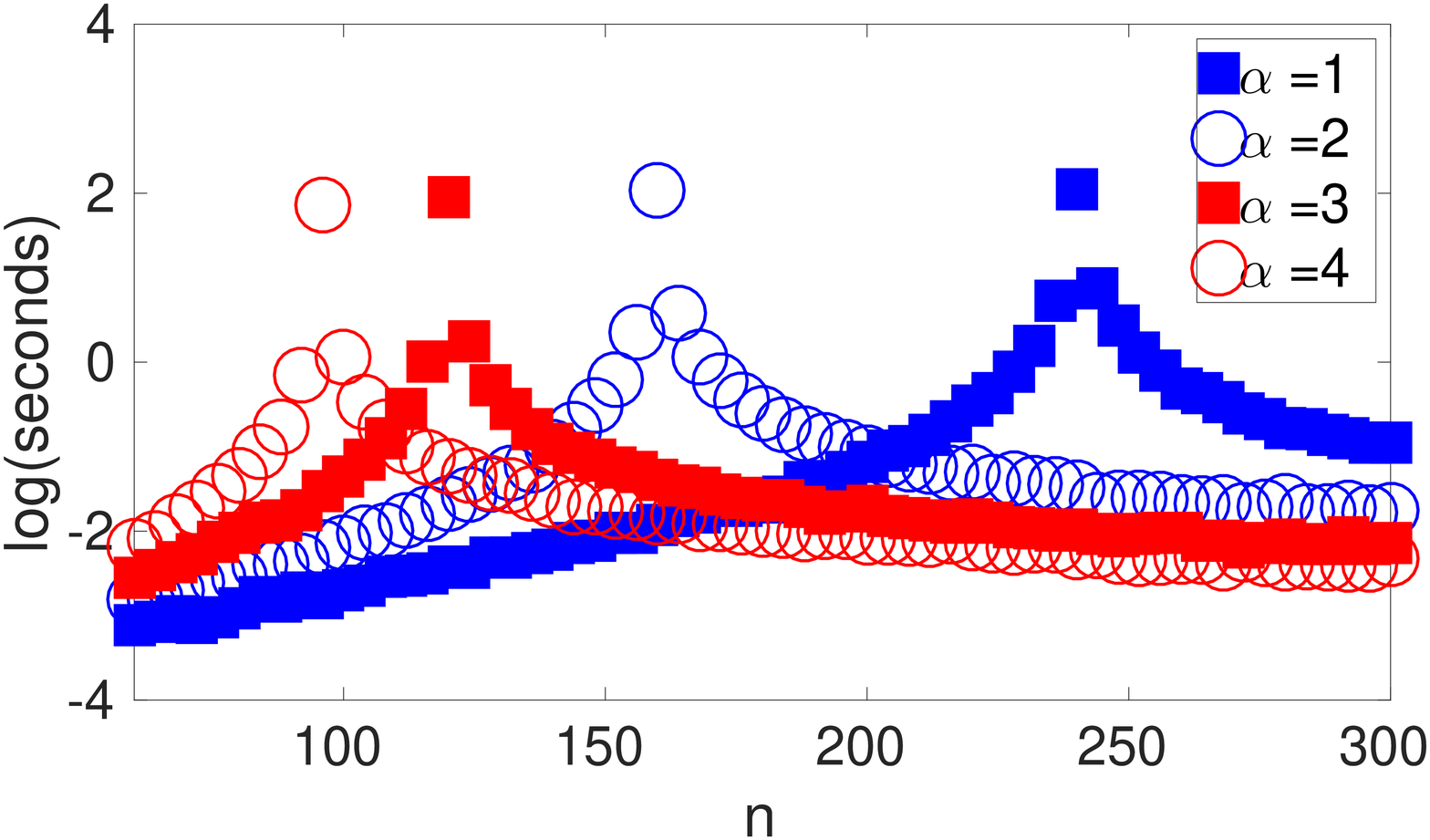}
\caption{LRE and log-runtime of th-RegTME  vs. number of samples \(n,\)
at $p =480$ and $\alpha = 1, 2, 3, 4$.}
\label{fig:TME_alpha}
\end{figure}

\subsection{Regularized TME in the presence of outliers} \label{sec:TME_outliers}

We conclude the numerical section with an illustrative example of the ability of the regularized TME to detect outliers, and upon their removal and thresholding, to provide a robust and accurate estimate of a sparse shape matrix. For a related rigorous study on the ability of Maronna's M-estimator to detect outliers, see
\cite{couillet_outs}. 

To this end, we consider the following $\epsilon$-contamination mixture model: $(1-\epsilon)n$ of the observed data, the inliers, follow an elliptical distribution with the same sparse shape matrix \(\Sp\) as above. The remaining $\epsilon n$ of the samples, the outliers,  follow an  elliptical distribution with shape matrix $\bm U(p \bm D/tr(\bm D))\bm U'$, where $\bm U$ is a unitary matrix, uniformly distributed with Haar measure, and $\bm D$ is a diagonal matrix.  In our first experiment, the diagonal entries \(d_{ii}\) are all i.i.d. uniformly distributed over $[1,5]$, so the outliers are rather diffuse. In our second experiment \(d_{11}=p,d_{22}=p/2\)
and all other $d_{ii}=1$, so the outliers are nearly on a 2-d randomly rotated subspace.

Given  $n$ samples from this $\epsilon$-contamination model, and without knowledge of $\epsilon$, the task is to accurately estimate the shape matrix $\Sp$. Since both the inliers and outliers have potentially heavy tailed distributions, it might not be possible to detect the outliers by simple schemes, such as those based on the norm of a sample or the number of its neighbors in a given radius. However, recall that by our theoretical analysis, in the absence of outliers ($\epsilon=0$),  the corresponding weights \(w_{i}\) in the regularized TME are all approximately equal. For $\epsilon\ll 1$, with all samples normalized to have unit norm, we thus expect  the inliers to still all have similar weights, and the outliers to have quite different weights, hopefully smaller though not necessarily so. With further details in Appendix \ref{sec:estimate_mu_sigma}, our proposed procedure for robustness to outliers is to estimate the mean and standard deviation of the inliers' weights. Then exclude all samples whose  weights are outside, say, the mean plus or minus two standard deviations, recompute the regularized TME on the remaining samples and threshold it.

Fig. \ref{fig:TME_outliers} illustrates the robustness of this procedure to outliers in two different settings. From left to right, for $\epsilon=0.2$ and $\epsilon=0.4$, it shows the weights of the $n$ normalized samples $\x_i/\|\x_i\|$, sorted so the first $\epsilon n$ of them are the outliers. The blue horizontal line is a robust estimate of the mean weight of the inliers, and the two red lines are this estimated mean plus and minus two standard deviations. The top row corresponds to the first outlier model with $d_{ii}\sim U[1,5]$. The second row corresponds to our second outlier model with  $D=\operatorname{diag}(p,p/2,1,\ldots,1)$. Note that this outlier shape matrix has a spectral norm $O(p)$, which does not satisfy our requirement that $\|\bm D\|\leq s_{\max}$. As indeed observed empirically, the weights of the outliers do not so tightly concentrate around some value. Yet, our outlier exclusion procedure still succeeds to exclude most of these outliers.
The error of the thresholded TME\ with outliers removed, compared to that of thresholding the original TME is shown in the right column of Fig. \ref{fig:TME_outliers}.

\begin{figure}
\includegraphics[width=0.31\textwidth]{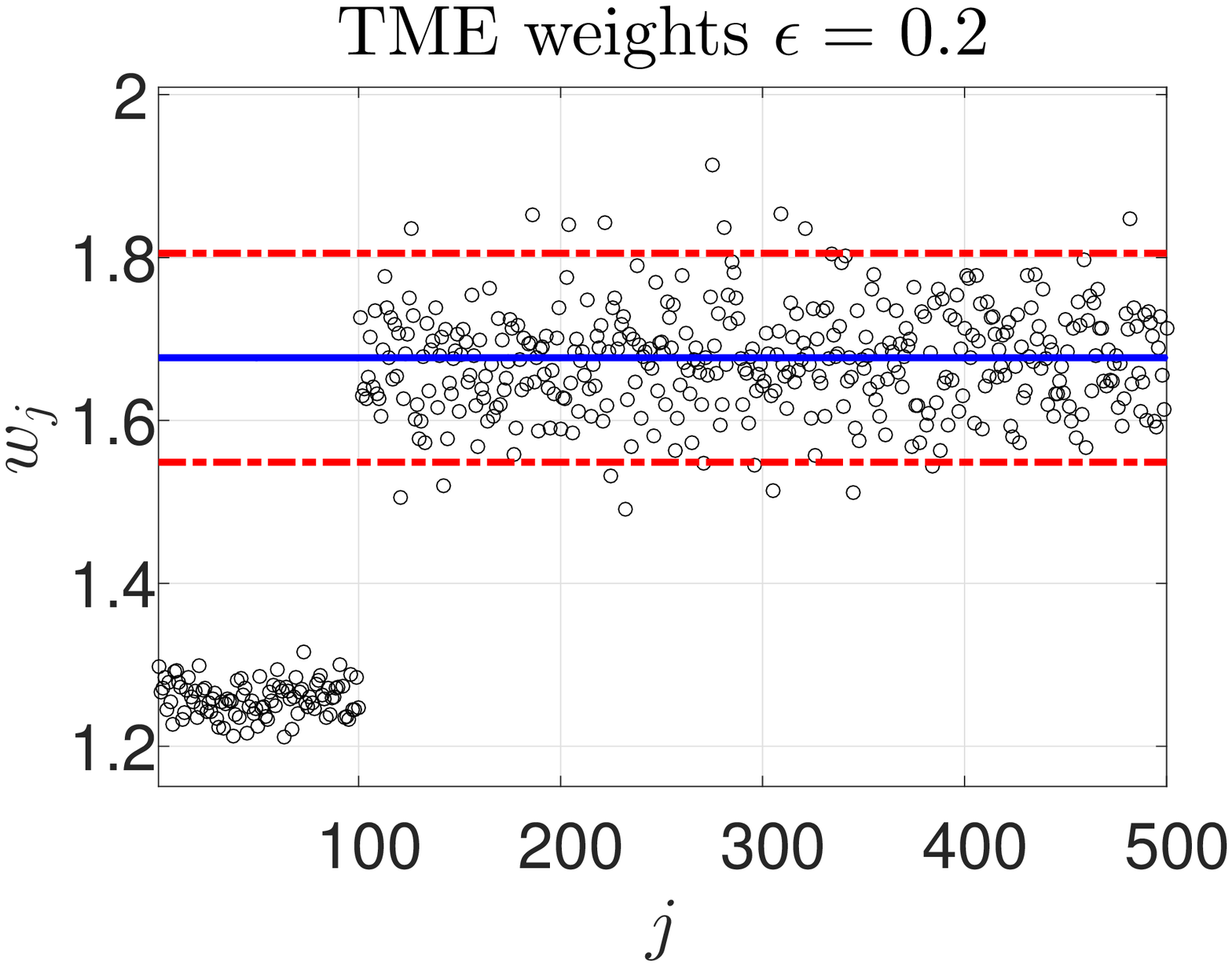}
\includegraphics[width=0.31\textwidth]{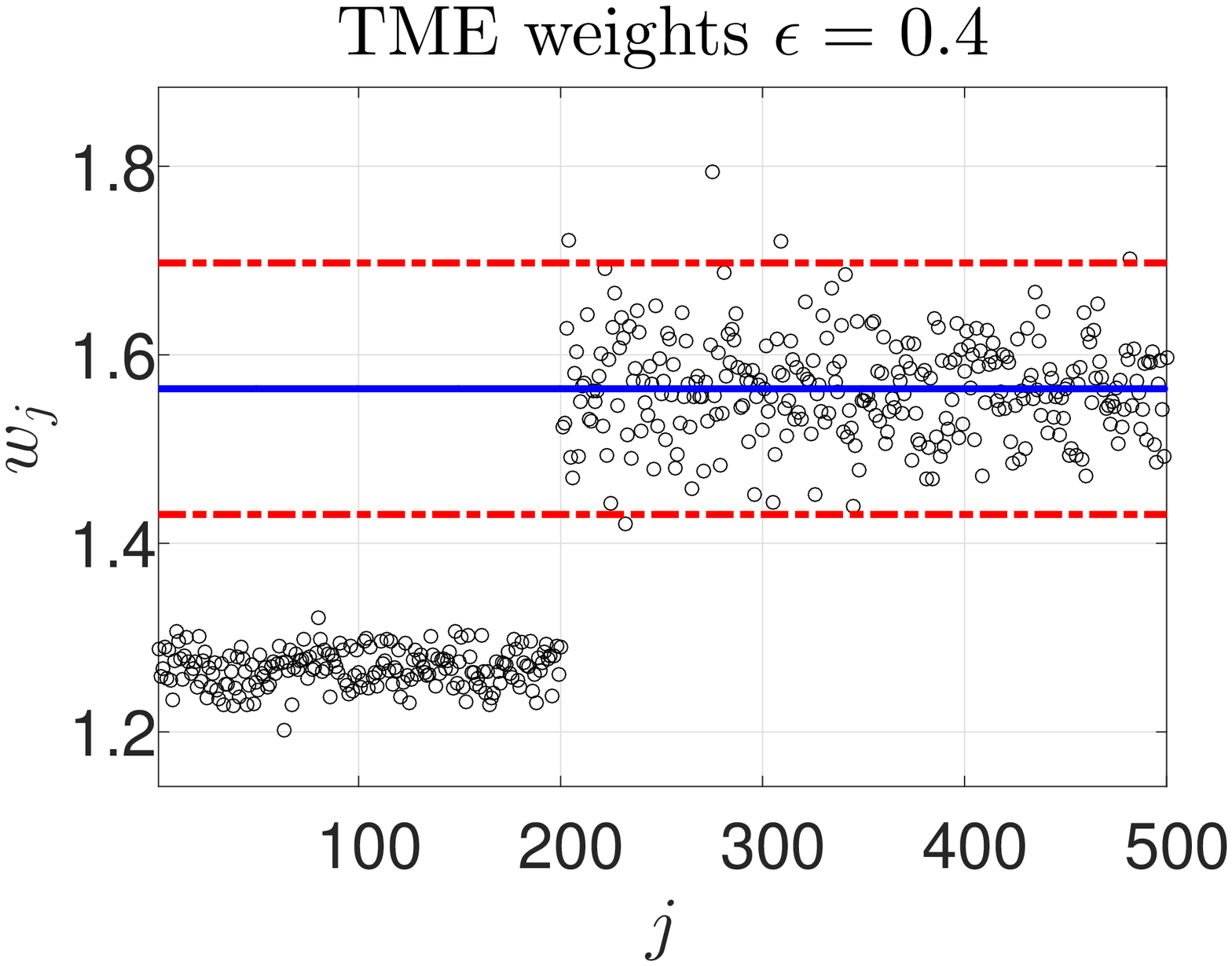}
\includegraphics[width=0.31\textwidth]{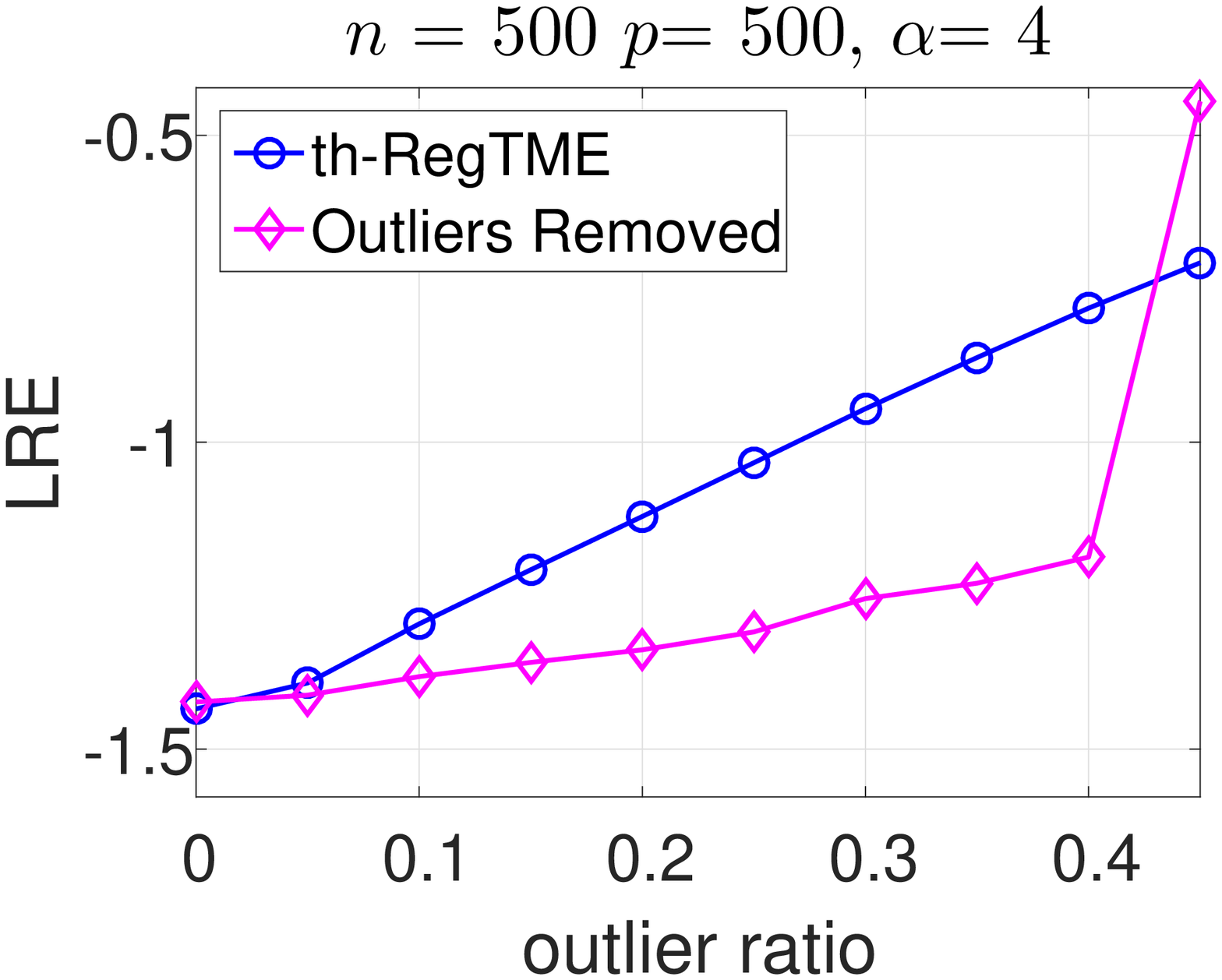}
\includegraphics[width=0.31\textwidth]{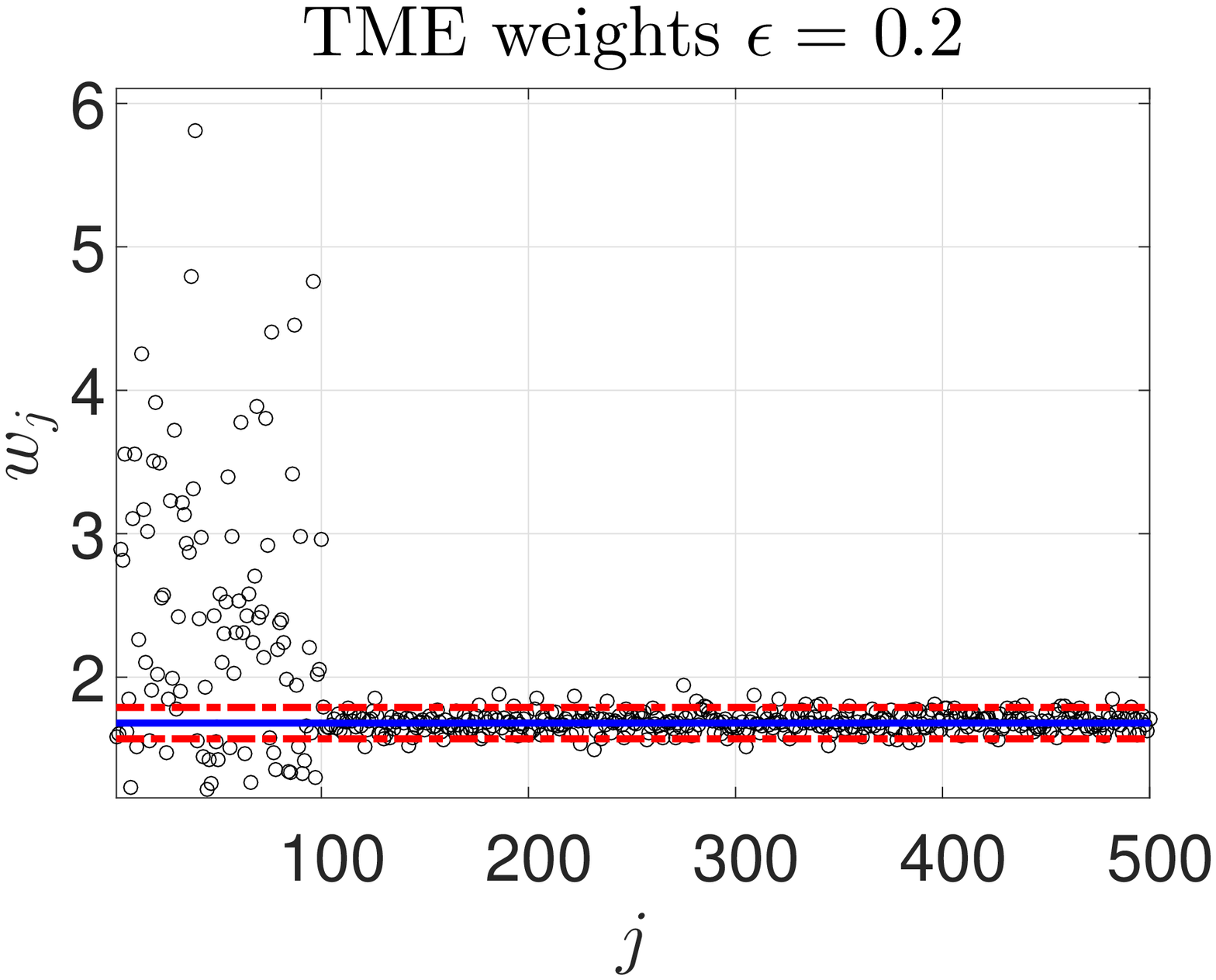}
\includegraphics[width=0.31\textwidth]{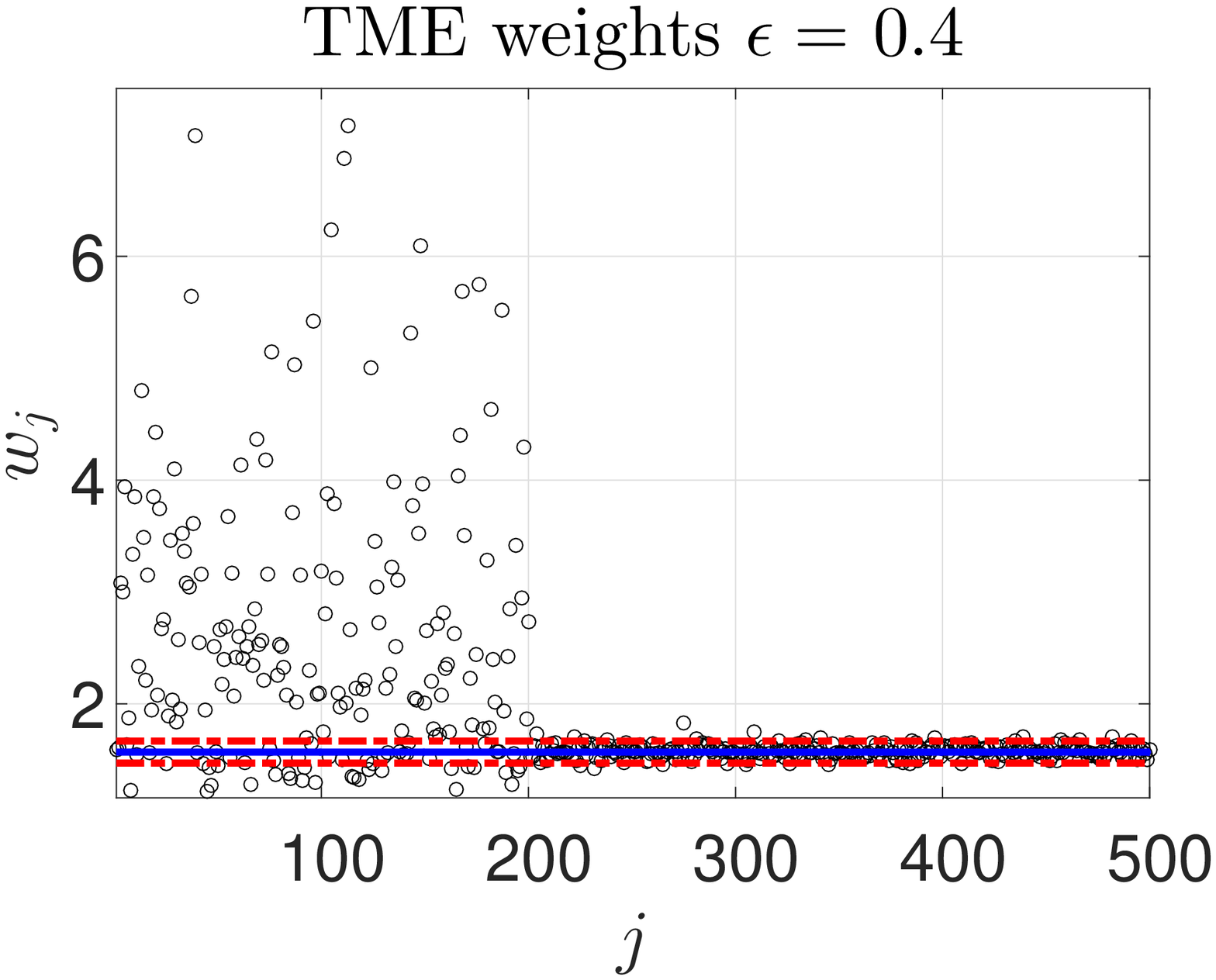}
\includegraphics[width=0.31\textwidth]{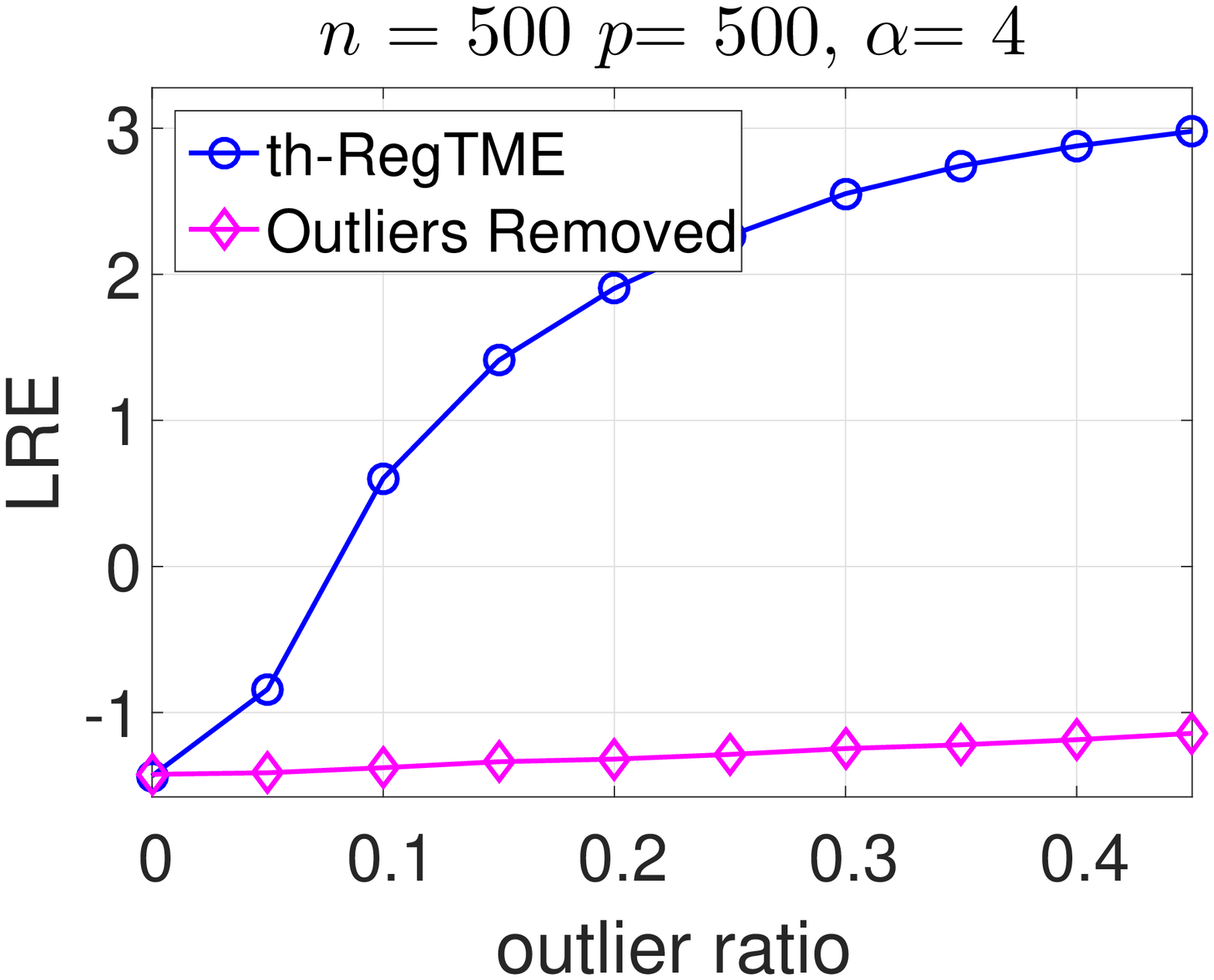}
\caption{The TME weights for $\epsilon=0.2,0.4$ and the log relative error (LRE)\ of thresholding the regularized TME, before and after outlier removal, vs. $\epsilon$. Top row \(D_{ii}\sim U[1,5]\). Bottom row $D=\operatorname{diag}(p,p/2,1,\ldots,1)$.}
\label{fig:TME_outliers}
\end{figure}

This simple example illustrates the potential ability of TME to screen outliers in high dimensional settings, at least for small contamination levels. A detailed study of this ability is an interesting topic for future research.

\section{Summary and Discussion} \label{Sec:disc}

In this paper we proposed simple estimators for the shape matrix of
possibly heavy tailed elliptical distributions, assuming the shape matrix is approximately sparse. We
further analyzed their error, showing that under the spectral norm they
are minimax rate optimal in a high-dimensional setting with \(p/n\to\gamma\).

There are several directions for future research. One direction is to extend our results to the case
\(p=n^{\beta}\), with $\beta>1$.
Our current analysis assumed the regularization parameter \(\alpha\) of TME is {\em fixed}, whereas if $p=n^\beta$ with\ $\beta>1$, just to ensure its existence would require $\alpha\to \infty$.
Handling this case thus requires extending our
analysis to allow $\alpha$ to grow with $n$
and $p$.

A question of practical interest is how to set the
threshold parameter in a data-driven fashion. \citet[Section 3]{CT}, proposed a cross validation procedure to set the threshold. Rigorously
proving that this provides a good estimate in the case of (regularized)\ TME is an interesting topic for future work.

While our work focused on approximate
sparsity of the shape matrix, robust inference under other common assumptions can also be studied.
For example, one might assume that the first few leading eigenvectors of
$\bm\Sigma$ are sparse, also known as sparse-PCA, or that $\bm\Sigma$ is the
combination of a low rank and a sparse matrix. In particular, a robust sparse-PCA
estimator may be constructed by applying a
sparse-PCA procedure to Tyler's M-estimator.

Finally, another direction for future work is to develop a computationally
efficient algorithm for sparse covariance estimation in the presence of a small
fraction of arbitrary outliers. This setting was considered in
\citet{matrixdepth}, but without a computationally tractable estimator.
Our promising preliminary results in Section \ref{sec:TME_outliers} suggest
to study whether regularized TME offers such robustness, and under which outlier models.


\section*{Acknowledgments}
We thank the three anonymous referees for multiple suggestions that greatly improved the manuscript. We also thank Teng Zhang and Ofer Zeitouni for useful discussions, and Tony Cai, Harrison Zhou, Elizaveta Levina and Peter Bickel for correspondence regarding their papers. This work was supported by NSF awards DMS-14-18386 and GRFP-00039202, UMN Doctoral Dissertation Fellowship and the Feinberg
Foundation Visiting Faculty Program Fellowship of the Weizmann Institute of
Science. We also thank the IMA and the schools of the authors for supporting collaborative visits.

\appendix

\section{Supplementary Details}


\subsection{Complexity of Calculating the Regularized TME}

\begin{proof}[Proof of Lemma \ref{lem:convergence_reg_TME}]


We arbitrarily fix a solution \(\bSa\) of \eqref{regTME}.
Since \(\bSa\) is invariant to scaling of the data, we assume that $\|\x_i\|=1$, $1 \leq i \leq n$.
We first analyze the quantity $e_1=\|\bSa-\bS_1(\alpha)\|$.
To this end, let $\lambda_{\max}=\|\bSa\|$. Taking the spectral norm in Eq. (\ref{regTME}), together with the fact that
$ \frac1{\x^T\bSa^{-1}\x}\leq \lambda_{\max}$ for any vector $\x$ with \(\|\x\|=1\),
\begin{eqnarray}
\lambda_{\max} &\leq & \tfrac{1}{1+\alpha}\left\|\tfrac{p}n\textstyle\sum_{i=1}^n \tfrac{\x_i\x_i^T}{\x_i^T\Sigma^{-1}\x_i}\right\|
        +\tfrac{\alpha}{1+\alpha} \nonumber
        \leq\tfrac{1}{1+\alpha}\lambda_{\max} C(\bm \tilde{\bm X}) +\tfrac{\alpha}{1+\alpha}.
\end{eqnarray}
Equivalently, for $1+\alpha>C(\bm \tilde{\bm X})$,
\begin{equation}
\lambda_{\max} \leq \frac{\alpha}{1+\alpha}\frac{1}{1-\frac{C(\bm \tilde{\bm X})}{1+\alpha}}.
       \nonumber 
\end{equation}
Combining this inequality with the fact that by Eq.~(\ref{regTME}) $\bSa - \frac{\alpha}{1+\alpha} \bI \in S_{+}^p$,
\begin{equation}
e_{1}= \Big\|\bSa-\frac{\alpha}{1+\alpha}\bI \Big\| = \lambda_{\max}-\frac{\alpha}{1+\alpha}\leq \frac{C(\bm \tilde{\bm X})}{1+\alpha}\frac{\alpha}{1+\alpha}
\frac{1}{1-\frac{C(\bm \tilde{\bm X})}{1+\alpha} }.
        \label{eq:bound_e1_CX}
\end{equation}

Next, we analyze the error $e_k$. We denote $\bm E_k=\bSa-\bS_k(\alpha)$ and write
\[
\bS_k(\alpha)=\bSa-\bm E_k=\bSa^{1/2}({\bm I}-\bSa^{-1/2}\bm E_k\bSa^{-1/2})\bSa^{1/2}.
\]
Since \(\bSa\) and \(\bS_k(\alpha)\) are invertible, so is
\(\bI-\bSa^{-1/2}\bm E_k\bSa^{-1/2}\). Let $\bm B_k={\bm I}- ({\bm I}-\bSa^{-1/2}\bm E_k\bSa^{-1/2})^{-1}$ and \(\bm R_{k}=\bSa^{-1/2}\bm B_k\bSa^{-1/2}\).
Then,
\[
\bS_k(\alpha)^{-1}=\bSa^{-1/2}(\bI -\bm B_{k})\bSa^{-1/2}
=\bSa^{-1}-\bm R_k.
\]
Subtracting Eq. (\ref{eq:iterations_reg_TME}) from Eq. (\ref{regTME}) gives
\begin{eqnarray}
\bm E_{k+1} & = &  \frac{1}{1+\alpha}\frac{p}{n}\sum_i \x_i\x_i^T\left(
    \frac{1}{\x_i^T\bSa^{-1}\x_i}-
\frac1{\x_i^T \bSa^{-1}\x_i-\x_i^T \bm R_k\x_i}
\right) \nonumber \\
& = &  \frac{1}{1+\alpha}\frac{p}{n}\sum_i
\frac{\x_i\x_i^T}{\x_i^T\bSa^{-1}\x_i} \left(1-\frac{1}{1-\delta_{ki}}\right), \nonumber
\end{eqnarray}
where $\delta_{ki}=\x_i^T \bm R_k \x_i/\x_i^T\bSa^{-1}\x_i$.

Let $D_k=\max_{1 \leq i \leq n}|\delta_{ki}/(1-\delta_{ki})|$. Since all terms $\x_i\x_i^T/\x_i^T\bSa^{-1}\x_i$ are positive semidefinite, the above equation implies that
\begin{equation}
\|\bm E_{k+1}\|\leq D_k\left\|\tfrac{1}{1+\alpha}\tfrac{p}{n}\textstyle\sum_i \tfrac{\x_i\x_i^T}{\x_i^T\bSa^{-1}\x_i}\right\|
        = D_k\left\|\bSa-\tfrac{\alpha}{1+\alpha}\bI \right\|
        = D_k e_1.
                \label{eq:bound_Ek_spectral}
\end{equation}
Eq. (\ref{eq:bound_e1_CX}) gives a bound on \(e_{1}\).
We now bound \(D_{k}\). Since \(\bSa\geq \frac\alpha{1+\alpha}\bI\),
\[
\|\bSa^{-1/2}\bm E_k\bSa^{-1/2}\|\leq \|\bSa^{-1}\| e_k \leq \frac{1+\alpha}{\alpha}e_k.
\]
Assume this quantity is strictly smaller than one, then
\begin{equation}
\label{eq:B_k}
\|\bm B_k\|=\|{\bm I}- ({\bm I}-\bSa^{-1/2}\bm E_k\bSa^{-1/2})^{-1} \| \leq\frac{1+\alpha}{\alpha}\frac{e_k}{1 -\frac{1+\alpha}{\alpha}e_k }.
\end{equation}
Finally, given the relation between $\bm R_k$ and $\bm B_k$,
$$
|\delta_{ki}|=\frac{|\x_i^T \bm R_k \x_i|}{\x_i^T\bSa^{-1}\x_i} =
\frac{|(\bSa^{-1/2}\x_i)^T \bm B_k (\bSa^{-1/2}\x_i)|}{\|\bSa^{-1/2}\x_i\|^2} \leq \|\bm B_k\|.
$$
Thus, assuming \(\|\bm B_k\|< 1\),
\begin{equation}
D_k = \max_i\frac{|\delta_{ki}|}{1-\delta_{ki}} \leq \frac{\|\bm B_k\|}{1-\|\bm B_k\|}
    = \frac{1+\alpha}{\alpha}e_k \cdot \frac{1}{1-2\frac{1+\alpha}{\alpha}e_k}.
        \label{eq:bound_Dk}
\end{equation}
Inserting (\ref{eq:bound_Dk}) and (\ref{eq:bound_e1_CX}) into (\ref{eq:bound_Ek_spectral}) yields that
\begin{equation}
\label{eq:ratios_ek}
\frac{e_{k+1}}{e_k} \leq \frac{C(\bm \tilde{\bm X})}{1+\alpha} \frac{1}{1-\frac{C(\bm \tilde{\bm X})}{1+\alpha}}
\frac{1}{1-2\frac{1+\alpha}{\alpha}e_k}.
\end{equation}

For the proof to hold, we required that $\frac{1+\alpha}\alpha e_k <1$ and $\|\bm B_k\|<1$. If $\frac{1+\alpha}\alpha e_k <0.5$,
then the RHS of Eq.~\eqref{eq:B_k} is less than one and both assumptions hold.
For $0<R<1$ and $1+\alpha > (3+R^{-1}) C(\bm \tilde{\bm X})$, Eq.~\eqref{eq:bound_e1_CX} implies that
$\frac{1+\alpha}{\alpha}e_1<\frac{1}{2+R^{-1}}$ and combining this with Eq.~\eqref{eq:ratios_ek} results in the estimate
$e_2/e_1 < R$. Since $R<1$, easy induction implies that for $k > 1$, $\frac{1+\alpha}{\alpha}e_k<\frac{1}{2+R^{-1}}< 0.5$, as required, and so
Eq.~(\ref{eq:linear_convergence_TME}) holds.
Since this convergence holds with any solution of \eqref{regTME}, this solution thus has to be unique.
\end{proof}

\begin{proof}[Proof of Lemma \ref{lem:CX_bound}]
    Since the regularized TME is invariant to scaling, we may assume that all $u_i\sim \chi^2_p$, and
    express $\bm x_i = \Sp^{\frac12}\bm\xi_i$, where $\bm\xi_i\sim N(\bm 0,\bm I).$
Let  $\bm U\bm D\bm U^T$ be the eigendecomposition of $\Sp$. Redefining
$\bm\xi = \bm U \bm\xi$, then
$\|\bm x_i\|^2 = \bm \xi_i^T \bm D\bm \xi_i$ and
\[
    C(\bm \tilde{\bm X}) = \left\|\Sp^{\frac12}\left(\frac 1n \sum_{i=1}^n
    \frac{\bm \xi_i\bm \xi_i^T}{\frac1p\bm \xi_i^T\bm D\bm \xi_i}\right)
    \Sp^{\frac12}\right\|.
\]
Combining Lemma \ref{Rudelson} with a union bound yields
$$\operatorname{Pr} \left( \max_i \left|\frac1p\bm \xi_{i}^T \bm D\bm \xi_i - \frac1p
\operatorname{tr}(\bm D)\right|  > \epsilon \right) < 2 n
\exp\left(-c_1 \min\left\{\frac{c_2^2 p^2\epsilon^2}{\|\bm
D\|_F^2},\frac{c_2 p\epsilon}{\|\bm D\|}\right\}  \right).
$$
Since $\|\bm D\|=\|\Sp\|$
and $\|\bm D\|_F^2 \leq p\|\Sp\|^2$,  for any fixed $\epsilon$ the above probability is exponentially small in \(p\). Taking say $\epsilon=1/2$ and recalling that
$\operatorname{tr}(\bm D)=p$,
gives that with high probability,
\[
    C(\bm \tilde{\bm X})\le 2 \|\Sp\|\cdot \Big\|
    \frac 1n\sum_{i=1}^n \bm \xi_i\bm \xi_i^T
    \Big\|.
\]
Eq. (\ref{eq:bound_CX}) follows since by Lemma \ref{lem:davidson_bound}, w.h.p. $\|\frac1n\sum_{i=1}^n
\bm \xi_i\bm \xi_i^T \| \le (1+2\sqrt{p/n})^2$. \end{proof}

\subsection{Proof of Lemma \ref{bl}}
\label{sec:proof_b1}
Most of the proof follows  \citet[p.~2583]{CT}. By the triangle inequality,
\bea
\|\tau_{t}(\bm A) - \bm B\| \le \|\tau_{t}(\bm B) - \bm B
\| + \|\ \tau_{t}(\bm A) - \tau_{t}(\bm B)\| = q_{1} + q_{2}.
        \nonumber
\eea
As in their Eq. (13),
$
q_1\le
t^{1-q}s_p$. For the second term $q_2$,
\bea
q_{2}&\le& \max_i\sum_{j=1}^p|a_{ij}|\bm 1(|a_{ij}|\ge t, |b_{ij}|<t) +
\max_i\sum_{j=1}^p |b_{ij}|\bm 1(|a_{ij}|< t, |b_{ij}|\ge t) \nonumber
\\ \nonumber
&&+
\max_i\sum_{j=1}^p |a_{ij}-b_{ij}|\bm1(|a_{ij}|\ge t, |b_{ij}|\ge t) =
q_{3}+q_{4}+q_{5}.
\eea
Similarly,  $q_4\leq C_1\sqrt{\frac{\log p}{n}} t^{-q}s_p+t^{1-q}s_p$ and $q_5\leq C_1\sqrt{\frac{\log p}{n}} t^{-q}s_p$.
For   \(q_{3}\),
$$
q_{3} \le \max_i\sum_{j=1}^p|a_{ij}-b_{ij}|\bm 1(|a_{ij}|\ge t, |b_{ij}|<t) +
\max_i\sum_{j=1}^p|b_{ij}|\bm 1(|a_{ij}|\ge t, |b_{ij}|<t).
$$
The second sum is bounded as above by \(t^{1-q}s_p\). For the first sum, we slightly differ from \cite{CT}. Since  $|a_{ij}-b_{ij}|\leq C_1\sqrt{\frac{\log p}n}$ and   $t=K\sqrt{\frac{\log p}n}$
with $K>C_1$ then all terms  satisfy \(|b_{ij}|>t(1-C_1/K).\) Hence,
\bea
q_3 &   \le & C_1\sqrt{\tfrac{\log p}{n}}\max_i \sum_{j=1}^p
\bm1\left(|b_{ij}|>t\left(1-\tfrac{C_1}K\right)\right) + t^{1-q}s_p
            \nonumber    \\
&\le &  C_1\sqrt{\tfrac{\log p}{n}} t^{-q}(1-\tfrac{C_1}K)^{-q}s_p+t^{1-q}s_p.
        \nonumber
\eea
Collecting the above inequalities concludes the proof, since
\bea
\|\tau_{t}(\bm A) - \bm B\| \le \left(3 K^{1-q}+ C_1K^{-q}(2+(1-C_1/K)^{-q}) \right) s_p \left(\frac{\log
p}{n}\right)^{\frac{1-q}2}.
                \nonumber
        \eea


\subsection{Proof of Lemma~\ref{lem:weights}} \label{weightspf}
Let $\x_i=\Sp^{1/2}\y_i$ where $\y_i$ are i.i.d. $N(0,\bm I)$. Let
$\hat{\bm\Sigma}_x$ and $\hat{\bm\Sigma}_y$ be their TME's, respectively.
It follows directly from Eq. \eqref{uniquewts} that the weights $\{w_i^x\}_{i=1}^n$ and
$\{w_i^y\}_{i=1}^n$ are identical. The latter are tightly concentrated around $1/n$ by \citet[Lemma 2.2]{MP}.

\subsection{Proof of Lemma~\ref{spec}} \label{pfspec}

Since $\x_i\sim N({\bm 0},\Sp),$  by Lemma \ref{lem:weights} the TME weights $\w=(w_1,\ldots,w_n)^T$ of Eq. (\ref{uniquewts}) are all concentrated around $1/n$. The following lemma shows that $T_w=\tr(\sum_i w_i \x_i\x_i^T)$ is close to $\tr(\Sp)=p$.
%
\begin{lem} \label{lem:p_over_Tw}
Assume the setting of Lemma~\ref{spec}. There exist constants \(C,c\) and \(c'<1\) depending on $\gamma$\ such that  $\forall\epsilon\in(0,c')$ and $n$\ sufficiently large,
\begin{equation}
        \Pr\left(\left|\tfrac{p}{T_{w}}-1\right|>\epsilon\right) \leq C n e^{-cn\epsilon^2}.
                \label{eq:Tw_over_p}
\end{equation}
\end{lem}
We  prove Lemma \ref{spec} assuming Lemma \ref{lem:p_over_Tw} holds,
and then prove the latter.

\begin{proof}[Proof of Lemma \ref{spec}]
By definition,
\begin{equation}
\| \bm{\hat\Sigma} - \bm{\hat S}\|_{\max} =
        \Big\| \sum_{i=1}^n \left( \tfrac{pw_i}
{ T_w} - \tfrac1n \right) \bm{x_i x_i}^T \Big\|_{\max} 
        \leq 
       \left\|\tfrac{n p {\bm w}}{T_w}-{\bm 1}\right\|_\infty \cdot
\Big\| \tfrac1n\sum_{i=1}^n \bm{x_i x_i}^T\Big\|_{\max}.
    \label{eq:pSigma_diff}
\end{equation}
Since \(\x_i\sim N(0,\Sp)$ with $\Sp\in\mathcal U(q,s_p,M)$, then w.h.p., \(\|\frac1n\sum \x_i\x_i^T\|_{\max}\leq 2M$.  
As for the first term on the RHS\ of Eq.~\eqref{eq:pSigma_diff},
by the triangle inequality,
\begin{equation}
\left\|\tfrac{n p {\bm w}}{T_w}-{\bm 1}\right\|_\infty
        =
\left\|\tfrac{n p {\bm w}}{T_w}-n\w+n\w-{\bm 1}\right\|_\infty
        \leq
        \|n\w\|_\infty\left|\tfrac{p}{T_w}-1\right|  +    \|n\w-{\bm 1}\|_\infty.
                \nonumber
\end{equation}
Hence,
\begin{equation}
\Pr\left(\left\|\tfrac{n p {\bm w}}{T_w}-{\bm 1}\right\|_\infty > \epsilon \right)
        \leq
        \Pr(\|n\w\|_\infty|\tfrac{p}{T_w}-1| > \epsilon/2)
        +
        \Pr(\|n\w-{\bm 1}\|_\infty >\epsilon/2).
                \nonumber
\end{equation}
Lemma \ref{lem:weights} provides an exponential bound on the second term. For the first term, applying Eq.~(\ref{eq:ABc_lambda}) with \(\lambda=2\) gives
\begin{eqnarray}
\Pr(\|n\w\|_\infty |\tfrac{p}{T_w}-1| > \epsilon/2) &\leq &\
                \Pr(\|n\w\|_\infty > 2) + \Pr(|\tfrac{p}{T_w}-1| > \epsilon/4)         \nonumber \\
        & \leq &
                \Pr(\|n\w-{\bm 1}\|_\infty > 1) + \Pr(|\tfrac{p}{T_w}-1| > \epsilon/4).
        \nonumber
\end{eqnarray}
By Lemmas \ref{lem:weights} and \ref{lem:p_over_Tw}, these two probabilities
are exponentially small.
\end{proof}

\begin{proof}[Proof of Lemma \ref{lem:p_over_Tw}]
As
$
|\tfrac p{T_w}-1| =  \tfrac{p}{T_w}|1-\tfrac {T_w}p|$,
by Eq.~(\ref{eq:ABc_lambda}) with $\lambda=2$
\bea
\Pr\left(\tfrac
p{T_w}|1-\tfrac {T_w}p|>\epsilon \right)
&\le& \Pr\left(\tfrac p{T_w}>2 \right) +
    \Pr\left(|1-\tfrac{T_w}p|>\epsilon/2 \right)
    \nonumber \\
&\le& \Pr\left(|\tfrac{T_w}p-1|>1/2 \right) + \Pr\left( |1-\tfrac{T_w}p|>\epsilon/2\right)
    \nonumber \\
&\le& 2\Pr\left( |1-\tfrac{T_w}p|>\epsilon/2\right).
    \nonumber
\eea
Next, we relate $|1-\tfrac{T_w}p|$
to $|1-\tfrac Tp|$, where $T=\tr (\frac1n\sum_{i=1}^n \bm x_i \bm x_i^T )$. Note that
\begin{eqnarray}
|1-\tfrac{T_w}p | &\le & |1-\tfrac{T}p| + |\tfrac{T}p - \tfrac{T_{w}}p|
=|1-\tfrac Tp|+ \tfrac1p\left| \sum_{i=1}^n (\tfrac1n-w_i)\bm x_i^T \bm x_i \right|   \nonumber \\
&\le & |1-\tfrac Tp| + \|n \w - \bm 1\|_\infty \cdot |\tfrac Tp|.
    \nonumber \label{eq:Tw_to_T}
\end{eqnarray}
Therefore
$$
\Pr\left( |1-\tfrac{T_w}p|>\tfrac\epsilon2\right) \le
\Pr\left(|1-\tfrac Tp|>\tfrac\epsilon4
\right)+  \Pr\left( \|n \w - \bm 1\|_\infty
\cdot |\tfrac Tp|>\tfrac\epsilon4\right)
    \nonumber
= q_1+q_2.
$$
Applying Eq.~(\ref{eq:ABc_lambda}) with $\lambda = 2$ to the second term gives
\bea
q_2 &\le& \Pr\left( \|n \w - \bm 1\|_\infty > \epsilon/8 \right)  +
\Pr\left(\tfrac Tp>2\right)
    \nonumber \\
&\le& \Pr\left( \|n \w - \bm 1\|_\infty > \epsilon/8 \right)  +
\Pr\left(|1-\tfrac Tp|>1\right).
    \nonumber
\eea
By Lemma~\ref{lem:weights}, the first probability above has the desired exponential decay. To conclude the proof, we thus need to provide an exponential bound on \(q_1\).


Let \(\lambda_1\geq \lambda_2\geq\ldots\geq \lambda_p\) be the eigenvalues of \(\Sp\). Since $\x_i\sim N({\bm 0},\Sp)$,
\bea
T=\tr\left(\frac1n\sum_{i=1}^n \bm{x}_i \bm{x}_i^T\right)&=& \sum_{j=1}^p
\lambda_j \chi_j^2(n)/n,
    \nonumber \label{trace_eigs}
\eea
where the $\chi^2_j(n)$ are i.i.d.~chi-square random variables with $n$ degrees
of freedom for $j=1,2,\dots,p$. Given that
$\tr(\bm S_p) = \sum_{j=1}^n \lambda_j = p$,
\bea
\left|1-\frac{T}{p}\right| =
\frac 1p \left|\sum_{j=1}^p \lambda_j \left(1-\frac{\chi_{j}^2(n)}{n}\right)   \right|
\le
\max_j \left|\frac{\chi_{j}^2(n)}{n}-1\right|.
        \nonumber
\eea
Since $\chi^2$ random variables are sub-exponential, for a suitable constant \(c>0,\)
\bea
{\rm {Pr}}\left(\bigg|\frac{\chi^2(n)}{n}-1\bigg|>\epsilon
\right)<\exp\left(-cn\epsilon^2 \right).
         \label{eq:chibound}
\eea
Therefore by a union bound,
the term \(q_{1}\) is also exponentially small.
\end{proof}

\subsection{Proof of Proposition~\ref{prop:r_finite}}
        \label{r_exists_pf}
%
     To prove the existence of a unique $r^{*}=r^*(p,n,\alpha,\Sp)$
which satisfies Eq.~\eqref{eq:EQ},
    we first show that $\mathbb E [Q(r)]$ is strictly monotone increasing in $r$ and
    then use the intermediate value theorem. 

    To simplify notation, let \(\bm T = \frac1n\sum_{j=1}^{n-1}\bm \xi_j \bm \xi_j^T\) and $\beta=\beta(r)=\frac{n\alpha}{pr}$. Then
 \bea
 \label{eq:EQr_def}
 \mathbb E[Q(r)] = \mathbb E_{ \bm\xi_i  }[\mathbb E_{\bm y}[Q(r)]]
     = \mathbb E \left[    \frac1p    \tr\left(  \left(\bm T + \beta \bm S_{p}^{-1} \right)^{-1}\right)   \right],
    \eea
where the expectation is now only over the random variables \(\bm \xi_i\).

First, we show that for any fixed $\Sp\in\mathcal{S}_p^{++}$, $\E[Q(r)]$ is strictly monotone
increasing in $r$. Indeed, taking the derivative with respect to $r$ and using the identity
$\tr(\bm A \bm B) = \tr(\bm B \bm A)$ gives that
\begin{equation}
\frac{d}{dr} \E[Q(r)] = \frac{n\alpha}{pr^2} \,\E\left[\frac1p \tr( (\bm T+\beta \Sp^{-1})^{-2} \Sp^{-1})\right].
        \nonumber
\end{equation}
Applying \citet{bhatia2013matrix}[Prob. III.6.14] and Jensen's inequality, 
\begin{eqnarray}
\frac{d}{dr} \E[Q(r)]  & \geq  & \frac{n\alpha}{pr^2}\lambda_{\min}(\Sp^{-1})
\E[\frac1p \tr((\bm T+\beta \Sp^{-1})^{-2})]
                                \nonumber \\
&  \geq & \frac{n\alpha}{pr^2}\frac1{s_{\max}}\E[1/\lambda_1(\bm T + \beta \Sp^{-1})^2] 
                        \nonumber \\
& \geq & \frac{n\alpha}{pr^2}\frac1{s_{\max}}
        \frac{1}{\E[ \lambda_1(\bm T + \beta \Sp^{-1})]^2}.
                        \nonumber 
\end{eqnarray}
Clearly, $\lambda_1(\bm T+\beta\Sp^{-1})\le \lambda_1(\bm
T)+\beta/{\lambda_{\min}(\Sp)}$. Furthermore, upon averaging over the
random variables $\bm \xi_i$, by Lemma \ref{lem:davidson_bound}, 
\(\E[\lambda_1(\bm T)] \leq (1+\sqrt{p/n})^2$. Therefore, for any fixed $\Sp$, the derivative of $\E[Q(r)]$ is
strictly positive for any $r>0$. Hence if there exists a solution 
to Eq. (\ref{eq:EQ}), then it must be unique. 

Next, we show that this solution must satisfy $r\geq r_{\min}$. By definition,
\bea
\E[Q(r)] &=& \frac1p \sum_{j=1}^n \frac{1}{\lambda_j(\bm T + \beta\bm S_p^{-1})} 
        \nonumber \\ 
        &\le&
         \frac1p \sum_{j=1}^n \frac{1}{\lambda_j(\beta\bm S_p^{-1})} = \frac1\beta \frac1p\sum_{j=1}^p \lambda_j(\bm S_p) = \frac{n\alpha}{p r}.
    \label{eq:Qr_lower_bound}
\eea
Combining (\ref{eq:Qr_lower_bound}) with  (\ref{eq:EQ}) implies that
\bea
r^{*}(p,n,\alpha,\Sp)\ge \frac{n}p\frac{\alpha }{1+\alpha-p/n} =r_{\min }.
    \nonumber \label{expect1}
\eea

Finally, we bound $r$ from above. To this end, note that
\bea
\E[Q(r)] & = & \frac1p \E\left[\sum_j \frac{1}{\lambda_j(\bm T + \beta \Sp^{-1})}\right]
         \geq  \frac1p\E\left[\sum_j \frac1{\beta/\lambda_j(\Sp) + \|\bm T\|}\right]
        \nonumber \\
        & \geq & \E\left[\frac{1}{\beta + s_{\max} \|\bm T\|}\right]
        \geq \frac{1}{\beta + s_{\max} \E[\|\bm T\|]}
         \nonumber
\eea
where the last inequality is Jensen's inequality. 
By Lemma \ref{lem:davidson_bound}, $\E[\|\bm T\|]\leq (1+\sqrt{p/n})^2$. Hence, for $\alpha > p/n-1 +  s_{\max} (1+\sqrt{p/n})^2$, the solution to Eq. (\ref{eq:EQ}) satisfies
\[
r^* \leq r_{\max} = \frac{n}{p} \frac{\alpha}{1+\alpha - p/n -  s_{\max} (1+\sqrt{\gamma})^2}.
\] 


\subsection{Proof of Lemma~\ref{lem1}} \label{pflem1}
%
%

Let $\hat{\bm S} = \frac1n\sum_{k=1}^n\bm
x_k\bm x_k^T$ and  $\hat{\bm T}=\frac1n\sum_{k=1}^n\bm \xi_k\bm \xi_k^T$, where   $\bm x_i=\bm S_p^{\frac12}\bm \xi_i$ and
$\bm \xi_i\overset{iid}\sim N(\bm 0,\bm I)$.
Then, Eq.~\eqref{eq:gee} may be written as
\begin{eqnarray}
\frac 1rg(\bm u)_i &=& 1-\frac{1/(1+\alpha)}{\frac1p \bm x_i^T\left(\hat{\bm
S}+\beta \bm I\right)^{-1}\bm x_i}
= 1-\frac{1/(1+\alpha)}{\frac1p \bm \xi_i^T\left(\bm S_p^{-\frac12}\hat{\bm
S}\bm S_p^{-\frac12}+\beta \bm S_p^{-1}\right)^{-1}\bm
\xi_i} \nonumber\\
&=& 1-\frac{1}{1+\alpha}\frac{1}{\frac1p \bm \xi_i^T\bm E\bm
\xi_i},
    \label{eq:g_u_i}
\end{eqnarray}
where $\bm E = \left(\hat{\bm
T}\bm +\beta \bm S_p^{-1}\right)^{-1}$  and
$\beta=\alpha\frac np\frac 1r$.
The quadratic form $\frac1p \bm \xi_i^T\bm E\bm
\xi_i$ is difficult to analyze directly because
$\bm E$ depends on ${\bm \xi}_i$.
To disentangle this dependency, let
$\hat{\bm{T}}_{-i}=\frac1n\sum_{k\neq i}\bm \xi_k\bm \xi_k^T$, and
$\bm E_{-i}=(\hat{\bm T}_{-i}+\beta \bm S_p^{-1})^{-1}$.
As \(\bm E^{-1}\) and $\bm E_{-i}^{-1}$ differ by a rank-one matrix $\frac1n\bm \xi_i\bm \xi_i^T$, by the Sherman-Morrison formula,
\bea
\bm E=  \bm E_{-i} -   \frac1n \frac{\bm E_{-i}\bm \xi_i \bm \xi_i^T \bm
E_{-i}}{1+\frac1n\bm \xi_i^T\bm E_{-i}\bm \xi_i}.
    \nonumber
\eea
Therefore,
denoting by $Q_i$ the quadratic form
\bea
Q_i(r)\equiv Q_i=\frac1p \bm \xi_i^T\bm E_{-i}\bm \xi_i,
\label{eq:Qdef}
\eea
it follows that
\bea
{\frac1p \bm \xi_i^T\bm E\bm \xi_i}=
Q_{i}- \frac{\frac pn  Q_i^2}{1+\frac pn  Q_i}
=
\frac{ Q_i}{1+\frac pn  Q_i}.
 \label{eq:Qeq}
\eea
Plugging this expression into Eq.~\eqref{eq:g_u_i} gives
\bea
\frac1r g(\bm u)_i =
\frac{ Q_i(1+\alpha -\frac pn) - 1}{(1+\alpha) Q_i}
\label{eq:Qr_poly}.
\eea

Next, to establish a concentration
bound for $g(\bm u)_{i}/r$, we study the concentration of $Q_i$.
Since $\bm \xi_i\sim N({\bm 0},{\bm I})$ and is  independent
of $\bm E_{-i}$,
\bea
\mathbb E Q_i=\mathbb E \tr(\bm E_{-i})/p.
    \nonumber
\eea
We first show that $Q_i$ concentrates tightly around ${\tr}(\bm E_{-i})/p$ in view of
concentration of quadratic forms. We then show that
${\tr}(\bm E_{-i})$ concentrates tightly around its mean using results about
the concentration of certain functions of the eigenvalues of random matrices.

Applying Lemma \ref{Rudelson} with $\bm\xi=\bm\xi_i$ and viewing the matrix $\bm E_{-i}$ as fixed,
\[
\Pr\left(\left|Q_i - \frac1p\tr\left(\bm
E_{-i}\right)\right|> \epsilon \right)\le 2
\exp\left(-c_1\min\left\{\frac{c_2^2 p^{2}\epsilon^2}{\left\|\bm E_{-i}\right\|_F^2}
,\frac{c_2 p\epsilon}{\left\|\bm E_{-i}\right\|}  \right\}\right),
\]
where the above probability is only w.r.t. $\bm \xi_i$. Next, given that $\bm E_{-i} = (\bm T_{-i}+\beta\Sp^{-1})^{-1}$,
then $\|\bm E_{-i}\|   \le \frac{s_{\max}}{\beta}$ and
$\|\bm E_{-i}\|_F^2\leq p s_{\max}^2/\beta^2$.
Thus,
\bea
\Pr\left(\left|Q_i - \frac1p{\tr}\left(\bm E_{-i}\right)\right|>
\epsilon \right)\le C\exp\left(-cp\epsilon^2\right),
    \label{eq:Q_min_tr_E}
\eea
where now the probability is over all of the $\bm\xi_k$'s.

It remains to obtain a concentration inequality for ${\tr}\left(\bm
E_{-i}\right)/p$.
To this end,  consider the following \(p\times(n-1+p)\) matrix,
\[
{\bm Y} = \left(
\begin{array}{cccccc|ccc}
{\bm \xi}_1 &\cdots & {\bm \xi}_{i-1} & {\bm \xi}_{i+1} & \cdots &{\bm \xi}_n & \sqrt{n\beta}\Sp^{-1/2}
\end{array}
\right).
\]
By definition,
all entries of ${\bm Y}$ are independent, the first $p\times(n-1)$ are standard Gaussian random variables and the rest deterministic.
Then, by
\citet{gui}[Corollary 1.8b]\footnote{There is a typo in the original paper. In the
notation of their Corollary 1.8, $\bm Z$ should
be replaced with $\bm Z/(M+N)$.},
for any function
$h:\mathbb{R}\to\mathbb{R}$ such that
$h(x^2)$ is Lipschitz with constant \(L\), for any \(\delta>0\)\
\begin{equation}
\Pr\left(\tfrac{1}{K}\left|\tr h(\tfrac{{\bm Y}{\bm Y^T}}{K})-\mathbb{E}\tr(h(\tfrac{{\bm Y}{\bm Y^T}}{K}))\right|>\delta\right) \leq 2 \exp\left(-\tfrac{\delta^2K^2}{2L^2}\right)
        \label{eq:concentration_tr_h}
\end{equation}
where \(K=2p+n-1\) and for a symmetric matrix $\bm A$ with eigenvalues $\lambda_j$, $\tr h(\bm A)=\sum_j h(\lambda_j)$.

Since
${\bm Y}{\bm Y}^T=n(\hat{\bm T}_{-i}+\beta\Sp^{-1})=n\bm E_{-i}^{-1}$, consider the function \[
h(x)=\frac{n}{p}\cdot\frac1x
\]
for which  $\frac{1}{2p+n-1}\tr h({\bm Y}{\bm Y}^T/(2p+n-1))=\tr(\bm E_{-i})/p$. Next, note that
for sufficiently large $n$ and sufficiently small $\epsilon$
\[
\lambda_{\min}\left(\frac{\bm Y\bm Y^T}{2p+n-1}\right) = \frac{n}{2p+n-1}
\lambda_{\min}(\hat{\bm T}_{-i}+\beta \Sp^{-1}) \geq \frac{1}{2\gamma+1+\epsilon}\frac{\beta}{s_{\max}}=x_0.
\]
We thus apply the function \(h\) only in the interval \(x\ge x_0\). The Lipschitz constant of \(h(x^2)\) for $n$ sufficiently large is bounded by
\[
L\leq \left|\frac{d}{dx}h(x^2)\bigg|_{x=x_0}\right|\leq
        16 \frac{(\gamma+0.5+\epsilon)^3}{\gamma-\epsilon}
 \left(\frac{s_{\max}}{\beta}\right)^3
 \leq
 16 \frac{(\gamma+1)^3}{\gamma}
 \left(\frac{s_{\max}}{\beta}\right)^3.
\]
Hence, applying (\ref{eq:concentration_tr_h}), there exists a positive constant \(c\) that depends on \(\gamma,\alpha,r\) and \(s_{\max}\) such that
\bea
\Pr\left(\frac1p\left|\tr(\bm E_{-i}) - \E \tr (\bm E_{-i})  \right|>\delta\right) \leq 2 \exp\left(-c p^2\delta^2\right).
        \label{eq:tr_E_concentration}
\eea
Next, by the triangle inequality
\begin{equation}
\Pr(|Q_i-\tfrac{\E\tr(\bm E_{-i})}p|>\epsilon) \leq
         \Pr\left(|Q_i-\tfrac{\tr(\bm E_{-i})}p|>\tfrac\epsilon2\right)+
         \Pr\left(|\tfrac{\tr(\bm E_{-i})}p- \tfrac{\E\tr(\bm E_{-i})}p |>\tfrac\epsilon2\right)
        \nonumber
\end{equation}
Combining the above equation with Eqs. (\ref{eq:Q_min_tr_E}) and (\ref{eq:tr_E_concentration}),
implies that 
at the value of  \(r\) specified in Proposition~\ref{prop:r_finite}, for which  $\mathbb E\left[\tr(\bm
E_{-i})/p\right] = \frac{1}{1+\alpha-\frac pn}$, 
\bea
{\rm {Pr}}\left(\left| Q_i- \frac{1}{1+\alpha-\frac pn} \right| > \epsilon
\right)<C e^{-cp\epsilon^2}.
    \label{eq:Qtail}
\eea

We are finally ready to establish a concentration result
for $\frac1r g(\bm u)$. Combining Eq.~\eqref{eq:Qr_poly} and  a union bound over all \(p\) coordinates of $g$,
\bea
\Pr\left(\left\|\tfrac1r g(\bm u) \right\|_{\infty}  > \epsilon \right)
\le p
\Pr\left(\left| \tfrac1r g(\bm u)_i \right| > \epsilon     \right)
\le p\Pr\left(\left|\tfrac{Q_i(1+\alpha-p/n)-1}{(1+\alpha)Q_i}
\right| > \epsilon  \right). \nonumber
\eea
Applying Eq.~\eqref{eq:ABc_lambda} with $\lambda=1$ to the equation above gives
$$
\Pr\left(\left\|\tfrac1r g(\bm u) \right\|_{\infty}  > \epsilon \right) <
p\Pr\left(\left| Q_i(1+\alpha-\tfrac{p}n)-1) \right| >
\epsilon \right)+ p\Pr\left( (1+\alpha)Q_i < 1 \right).
$$

By Eq.~\eqref{eq:Qtail}, the first term on the RHS is exponentially small in \(p\). As for the second term, since \((1+\alpha)^{-1}<(1+\alpha-p/n)^{-1}\),
then again by Eq. (\ref{eq:Qtail}), $\Pr(Q_i<1/(1+\alpha))$ is also exponentially small in $p$.
The lemma thus follows from the boundedness of $r$ from above, as established in Proposition~\ref{prop:r_finite}.


\subsection{Proof of Lemma~\ref{lem2}} \label{pflem2}

For any \(\bm v\in\R^n\), denote $\hat{\bm S}(\bm v)=\frac1n\sum_k v_k\x_k\x_k^T$ and
\begin{equation}
\F(\bm v)=\left(\frac1r\hat{\bm S}(\bm v)+\beta\bm I\right)^{-1}.
    \label{Fdef}
\end{equation}
We  prove Lemma \ref{lem2} using the following Lemma, which is proved  latter.

\begin{lem}
\label{lemma:prop_F}
Assume the setting of Theorem~\ref{regthm2} and let $\bm u = r\bm1$, with $r$
defined in Proposition~\ref{prop:r_finite}. The matrix $\bm F$ of Eq.~\eqref{Fdef}
satisfies:
\begin{enumerate}
\item \label{propF1} For all $\bm v \in \R^n$ with $\|\bm v-\bm u\|_\infty \le r$,  $\|\F(\bm v)\|\leq 1/\beta$.
\item \label{propF2} There exists $c>0$ such that with probability at least $1-\exp(-cp)$, for all $\bm v \in \R^n$ with $\|\bm v-\bm u\|_\infty\le r$,
\begin{equation}
\lambda_{\min}(\F(\bm v)) \geq c_F = \frac{1}{2s_{\max}(1+2\sqrt{\gamma})^2+\beta}.
        \label{eq:lambda_min_F}
\end{equation}
\item \label{propF3} With the same  constant $c>0$ above, there exist $c',C>0$ such that
\begin{equation}
\Pr\left( \forall \bm v\in B_{c'}(\bm u), \|\bm F(\bm u) - \bm F(\bm v)\| < C\|\bm v - \bm
u\|_{\infty}\right)\ge 1-\exp(-cp).
    \label{eq:Diff_Fu_Fv}
\end{equation}
\end{enumerate}
\end{lem}

\begin{proof}[Proof of Lemma \ref{lem2}]
Recall that for
an invertible matrix $\bm A(\bm v)$ that depends on a vector
$\bm v$,
$
\frac{\partial(\bm A^{-1})}{\partial v_i} = -\bm A^{-1}\frac{\partial\bm
A}{\partial v_i}\bm A^{-1}.
$
Then, differentiating $g(\bm v)_i$ in Eq.~\eqref{eq:gee}
with respect to $v_i$ gives that $\nabla g(\bm v)=\bm I - \bm B (\bm v)$, where
\bea
\bm B(\bm v)_{ij} = \frac{1}{1+\alpha}\frac{p}{n}\frac{\F(\bm v)_{ij}^2}{\F(\bm v)_{ii}^2}, \qquad 1\leq i,j\leq n
    \label{eq:Grad_g}
\eea
and $\F(\bm v)_{ij}=\x_i^T\F(\bm v)\x_j$.
With this expression for $\nabla g(\bm v)$,
\begin{equation}
\|\nabla g(\bm u)-\nabla g(\bm v)\|_{\max}=\max_{i,j}
\frac1{1+ \alpha}\frac{p}{n}
\left|
\frac{\F(\bm v)_{ij}^2}{\F(\bm v)_{ii}^2}-\frac{\F(\bm u)_{ij}^2}{\F(\bm
u)_{ii}^2}\right|.
        \label{eq:diff_grad_g}
\end{equation}
By the triangle inequality,
\bea
\left|
\frac{\F(\bm v)_{ij}^2}{\F(\bm v)_{ii}^2}-\frac{\F(\bm u)_{ij}^2}{\F(\bm u)_{ii}^2}\right|
                & \leq &
\F(\bm v)_{ij}^2\left|
\frac{1}{\F(\bm v)_{ii}^2}-\frac{1}{\F(\bm u)_{ii}^2}\right|
+
\left|
\frac{\F(\bm v)_{ij}^2-\F(\bm u)_{ij}^2}{\F(\bm u)_{ii}^2}\right| \nonumber \\
&=& q_1+q_2.
    \nonumber
\eea
We now bound each of these two terms. For the first one,
\[
q_1 \leq\F(\bm v)_{ij}^2 \frac{
        |\F(\bm u)_{ii}-\F(\bm v)_{ii}|\cdot \left(\F(\bm u)_{ii}+\F(\bm v)_{ii}\right) }
{\F(\bm u)_{ii}^2\F(\bm v)_{ii}^2}.
\]
For any $\bm v$ for which \(\bm F(\bm v)\) is defined, $|\F(\bm v)_{ij}| \leq \|\bm F(\bm v)\|\|\x_i\|\|\x_j\|$ and $\F(\bm v)_{ii} \geq \lambda_{\min}(\F(\bm v)) \|\x_i\|^2$. Combining these with parts \ref{propF1} and \ref{propF2}
of Lemma~\ref{lemma:prop_F},
\[
q_1
\leq
\frac{1}{c_F^4} \frac{ \|\x_j\|^4}{\|\x_i\|^{4} }  (\|\F(\bm u)\|+
        \|\F(\bm v )\|)
\|\F(\bm u)-\F(\bm v)\|
\leq
\frac{2}{\beta c_F^4} \frac{ \|\x_j\|^4}{\|\x_i\|^{4} }
                 \|\F(\bm u)-\F(\bm v)\|
\]
and similarly
\[
q_2 \leq \frac2{\beta c_F^2}\frac{\|\x_j\|^2}{\|\x_i\|^2}\|\F(\bm u)-\F(\bm
v)\|.
\]
Finally, we write $\x_j=\Sp^{1/2}\bm \xi_j$ 
with $\bm \xi_j\sim N(0,\bI)$. Hence, $\|\x_j\|^2=\bm\xi_j^T\Sp\bm\xi$ is a quadratic form tightly concentrated around $\tr(\Sp)=p$. Therefore, w.h.p., $\|\x_j\|^2/\|\x_i\|^2 \leq 2$.
Next, Eq.~\eqref{eq:Diff_Fu_Fv} implies that w.h.p. $\| \F(\bm u)-\F(\bm v)\|\leq C\|\bm v-\bm u\|_\infty$.
A union bound on all $p^2$ terms in Eq.~(\ref{eq:diff_grad_g}) concludes the proof of the lemma.
\end{proof}

\begin{proof}[Proof of Lemma \ref{lemma:prop_F}]
Part \ref{propF1}: For any $\bm v \in \R^n$ with $\|\bm v-\bm u\|_\infty \le r$, all entries
$v_j\geq 0$, so  $\hat{\bm S}(\bm v) \in S_{+}^p$ and thus $\|\F(\bm v)\|\leq 1/\beta$.

Part \ref{propF2}: If $\|\bm v-\bm u\|_\infty \le r$, then \(v_{j}\leq 2r\) for all $1 \leq j \leq n$.
Thus,
\[
\lambda_{\min}(\F(\bm v)) \geq
        \frac1{\lambda_{\max}(\tfrac{1}{nr}\sum_k v_k\x_k\x_k^T)+\beta}
\geq
        \frac{1}{2\lambda_{\max}(\tfrac1n\sum_k \x_k\x_k^T)+\beta}.
\]
Eq.~\eqref{eq:lambda_min_F} follows since by Lemma~\ref{lem:davidson_bound}, with probability at least $1-\exp(-cp)$, the largest eigenvalue is smaller than $s_{\max}(1+2\sqrt{\gamma})^2$.

Part \ref{propF3}: Using the Hadamard product $\circ$, $\bm d_{\bm v}\in\{-1,1\}^n$ and $\bm \epsilon = \langle\epsilon_1,\dots,\epsilon_n\rangle$
with $\epsilon_1,\dots,\epsilon_p\ge 0$, we express $\bm v$ as
$\bm v = \bm u + r\bm d_{\bm v}\circ \bm \epsilon.$
Next, we apply the following classical
perturbation result \citep[Eq. (1.2)]{stewart}: Let \(\bm M\) be an invertible matrix $\bm M$, then for any
matrix $\Delta\bm M$ with $\|\bm M^{-1}\|\|\Delta\bm M\|<1$,
\bea
\left\|(\bm M+\Delta \bm M)^{-1}-\bm M^{-1}   \right\| \le \frac{\|\bm
M^{-1}\|^2 \|\Delta\bm M\|  }{1-\|\bm M^{-1}\| \,\|\Delta \bm M\| }.
    \label{eq:demm}
\eea
We use this inequality with $\bm M=\bm F(\bm u)^{-1} =  \hat{\bm S}(\bm 1)+\beta\bm I$ and $\Delta \bm M =
\hat{\bm S}(\bm d_{\bm v} \circ \bm\epsilon) $, so that $\bm F(\bm v)-\bm F(\bm u)= (\bm M+\Delta
\bm M)^{-1}-\bm M^{-1}$.

We first verify that the condition $\|\bm
M^{-1}\|\|\Delta\bm M\|<1$ holds.
Combining the non-negativity of the elements
of $\bm \epsilon$, the fact that for any $\bm P, \bm Q \in S_{+}^p$,
$\|\bm P-\bm Q\|\le\|\bm P+\bm Q\|$  and  Lemma~\ref{lem:davidson_bound},
we conclude that with probability $\ge1-\exp(-cp)$,
$$
\|\Delta \bm M\|
=\| \hat{\bm S}(\bm d_{\bm v}\circ \bm \epsilon)\|
\le \| \hat{\bm S}(\bm \epsilon) \|
\le    \| \hat{\bm S}(\bm 1) \| \cdot  \|\bm \epsilon\|_{\infty}
\le s_{\max}\left(1+2\sqrt{\gamma}\right)^2
\|\bm \epsilon\|_{\infty}.
$$
Next, by definition, 
$
\|\bm M^{-1}\| = \|\bm F(\bm u)\| \leq 1/\beta$.
Thus,
$\|\bm M^{-1}\|\|\Delta\bm M\| \le \tfrac1{\beta}s_{\max}\left(1+2\sqrt{\gamma}\right)^2\|\bm
\epsilon\|_{\infty}$. So, there is a constant $c'\le1$ such that with probability
$1-\exp(-cp)$, for all for
$\|\bm v-\bm u\|_{\infty} \le c'$, $\|\bm
M^{-1}\|\|\Delta\bm M\|<1$.
 Eq.~\eqref{eq:demm} and the definitions of $\bm M$ and $\Delta\bm M$ imply the desired bound.
\end{proof}


\subsection{Proof of Lemma~\ref{lem3}} \label{pflem3}
%
%
Recall that $\nabla g(\bm u)=\bm I-\bm B,$ with $\bm B$ given in
Eq.~\eqref{eq:Grad_g}.
Since $1+{\alpha}>\sup_n p/n$, then
$\operatorname{diag}(\bm B) =
\kappa \bm I$, with $\kappa= p/(n(1+\alpha))\in (0,1)$. Therefore, $\nabla g(\bm
u)=\left(1-\kappa\right)\bm I-\bm B_0$, where $\operatorname{diag}(\bm B_0)=\bm 0$, and
\bea
\|\left(\nabla g(\bm u) \right)^{-1}
\|_\infty&=&\tfrac{1}{1-\kappa}\left\|\sum_{k=0}^\infty
\Big(\tfrac{1}{1-\kappa}\Big)^k \bm
B_0^k \right\|_\infty \leq \sum_{k=0}^\infty \Big(\tfrac1{1-\kappa}\Big)^{k+1}\|\bm B_0\|_\infty^k.
\nonumber
\eea
Suppose that for some fixed $\lambda\in(0,1)$
\bea
\Pr\left( \|\bm B_0\|_{\infty} > \lambda(1-\kappa) \right) \le Cpe^{-cp}, \label{B0less}
\eea
then the lemma follows, since with probability at least \(1-Cpe^{-cp}\)
\[
\|\left(\nabla g(\bm u) \right)^{-1}\|_\infty
\leq \sum_{k=0}^\infty \left(\tfrac1{1-\kappa}\right)^{k+1} \bm
(1-\kappa)^k\lambda^k = \frac{1}{(1-\lambda)(1-\kappa)}.
\]

It suffices to prove Eq.~\eqref{B0less}.
To this end, from Eq.~(\ref{eq:Grad_g}), with \(\bm F=\bm F(\bm u)\) and $\hat{\bm S}_{-i}=\frac1n\sum_{j\neq i}\x_j\x_j^T$,
\bea
\sum_{j=1}^n(\bm B_0)_{ij} &=& \frac{1}{1+\alpha}\frac p n\sum_{j\neq i}^n
\frac{\left(\bm x_i^T\bm F \bm x_j\right)^2}
{\left(\bm x_i^T\bm F \bm x_i\right)^2}
=\frac{1}{1+\alpha}\frac p n \frac{\bm x_i^T\bm F\left(\sum_{j\neq i}^n \bm x_j\bm x_j^T\right)
\bm F \bm x_i}
{\left(\bm x_i^T\bm F\bm x_i\right)^2} \nonumber \\
&=& \frac{p}{1+\alpha}  \frac{\bm x_i^T\bm F\hat{\bm S}_{-i}\bm F \bm x_i}
 {\left(\bm x_i^T\bm F\bm x_i\right)^2}
= \frac{p}{1+\alpha} \frac{ A_1}{A_2}.
    \label{eq:Bnot}
\eea
Recall that \(\bm F=\bm F(\bm u)=(\hat{\bm S}+\beta\bm I)^{-1}\) and denote \(\bm F_{-i}=(\hat{\bm S}_{-i}+\beta\bm I)^{-1}\).
By the Sherman-Morrison formula, the numerator $A_1$ may be rewritten as
\bea
A_1&=&\bm x_i^T\bm \left(\bm F_{-i}-\frac1n\frac{\bm F_{-i}\bm x_i\bm x_i^T\bm F_{-i}}{1+\frac1n\bm x_i^T\bm F_{-i}\bm x_i}\right)\bm{\hat S}_{-i}\left(\bm F_{-i}-\frac1n\frac{\bm F_{-i}\bm x_i\bm x_i^T\bm F_{-i}}{1+\frac1n\bm x_i^T\bm F_{-i}\bm x_i}\right)\bm x_i
    \nonumber \\
&=& \bm x_i^T \bm F_{-i}\bm{\hat S}_{-i}\bm F_{-i}\bm x_i
- \frac2n \frac{\left(\bm x_i^T \bm F_{-i}\bm x_i\right)\left(\bm x_i^T \bm F_{-i}\bm{\hat S}_{-i}\bm F_{-i} \bm
x_i\right)}{1+\frac1n\bm x_i^T \bm
F_{-i}  \bm x_i}
    \nonumber \\
&&+\frac1{n^2}\left(\frac{1}{1+\frac1n\bm x_i^T \bm F_{-i}  \bm
x_i}\right)^2\bm \left(\bm    x_i^T\bm F_{-i}\bm x_i\right)^2\left( \bm x_i^T \bm
F_{-i} \bm{\hat S}_{-i} \bm F_{-i}\bm x_i \right).
    \nonumber
\eea
Next, recall that by
Eq.~\eqref{eq:Qdef}, with \(\x_i=\Sp^{1/2}\bm \xi_i\), it follows that  $Q_i = \frac1p\bm x_i^T
\bm F_{-i}\bm x_i$.
With  $R=\frac1p\bm x_i^T \bm F_{-i}\hat{\bm S}_{-i}\bm
F_{-i}\bm
x_i,$ the term \(A_{1}\) can be simplified to
\[
A_1 = pR\left(1 - \frac2n \frac{pQ_i}{1+\frac pn Q_i}
+\frac1{n^2}\left(\frac{pQ_i}{1+\frac pn Q_i}\right)^2 \right).
\]

Similarly, by Eq.~\eqref{eq:Qeq}, $A_2=\left( \bm x_i^T\bm F\bm
x_i\right)^2=(\bm\xi_i^T\bm E\bm \xi_i)^{2}=\left(\frac{p Q_i}{1+\frac pn Q_i}\right)^2$.
Thus,
\bea
\frac{p}{1+\alpha}\frac{A_1}{A_2}&=&\frac{\left(1+\frac pn Q_i \right)^2 p^2R
\left( 1-\frac2n\frac{p Q_i}{1+\frac pn Q_i} + \frac1{n^2}\left(\frac1{1+\frac
pn Q_i}\right)^2 p^2Q_i^2 \right)}{(1+\alpha) p^2Q_i^2}
    \nonumber \\
&=& \frac{R}{ Q_i^2}\frac{1}{1+\alpha}\left( \left(1+\frac pn Q_i \right)^2 -
\frac 2n pQ_i \left(1+\frac pn Q_i \right)+\frac{p^2}{n^2}  Q_i^2  \right)
    \nonumber \\
&=&  \frac{R}{ Q_i^2}\frac{1}{1+\alpha}  \left(\left(1+\frac pn Q_i\right)-\frac{p}{n}  Q_i\right)^2
    = \frac{R}{ Q_i^2}\frac{1}{1+\alpha}.
    \label{rq2}
\eea

Eqs. \eqref{eq:Bnot} and \eqref{rq2} give that 
$\sum_{j=1}^n (\bm
B_0)_{ij} = \frac{R}{ Q_i^2}\frac{1}{1+\alpha}$. Taking a union bound,
\bea
\Pr\left( \|\bm B_0\|_{\infty} > \lambda(1-\kappa) \right) \le  p\Pr\left(
\tfrac1{1+\alpha}\tfrac{R}{Q_i^2} > \lambda(1-\kappa) \right). \label{B0kappa}
\eea
To estimate the RHS of Eq.~\eqref{B0kappa}
we first show that
$R/Q_i<1$. Let $\bm U\bm D\bm U^T$, with $\bm D=\text{diag}(d_1,\dots,d_p)$ be the eigendecomposition of $\hat{\bm
S}_{-i}$. Then
\bea
\frac R{Q_i} &=& \frac{\bm x_i^T \bm F_{-i}\hat{\bm S}_{-i}\bm
F_{-i}\bm
x_i}{\bm x_i^T \bm F_{-i}\bm x_i} \nonumber
= \frac{\bm x_i^T  (\bm U\bm D\bm U^T +  \beta \bm I)^{-1}  \bm U\bm D\bm U^T (\bm
U\bm D\bm U^T  + \beta   \bm I)^{-1} \bm x_i}{\bm x_i^T   (\bm
U\bm D\bm U^T  + \beta   \bm I)^{-1}  \bm
x_i} \\ \nonumber
&=& \frac{(\bm U^T\bm x_i)^T   (\bm D+\beta\bm I)^{-1}    \bm D
(\bm D+\beta\bm I)^{-1}     (\bm U^T\bm x_i) }{(\bm U^T\bm
x_i)^T    (\bm D+\beta\bm I)^{-1}        (\bm U^T\bm x_i) } \\
&=&\frac{(\bm U^T\bm x_i)^T   \text{diag}\left(\frac{d_i}{(d_i+\beta)^2}\right)     (\bm U^T\bm x_i) }{(\bm U^T\bm
x_i)^T    \text{diag}\left(\frac{1}{d_i+\beta}\right)       (\bm U^T\bm
x_i) }\le
                \frac{d_1}{d_1+\beta},
    \nonumber \label{fracRQ}
\eea
where \(d_{1}=\|\bm S_{-i}\|\). By Lemma \ref{lem:davidson_bound}, with high probability $d_1<s_{\max}(1+2\sqrt{\gamma})^2$. Hence, there exists a  $\delta>0$, so that  w.h.p. $R/Q_i<1/(1+\delta)$. Let $\lambda =
\left(\frac{1}{1+\delta}\right)^2<1$, then by Eq.~\eqref{eq:Qtail},
\bea
&&\Pr\left( \frac{R}{(1+\alpha)Q_i^2}> \lambda(1-\kappa) \right) \le  \Pr\left(
Q_i(1+\alpha) < \frac{1+\delta}{1-\kappa} \right) \nonumber \\ \nonumber
&&\le \Pr\left(Q_i(1+\alpha)-\frac{1}{1-\kappa} <  \frac{\delta}{1-\kappa}
\right)\leq  Ce^{-cp}.
    \nonumber
\eea
Combining the above with Eq.~\eqref{B0kappa} implies that Eq.~\eqref{B0less} holds, as desired.

\subsection{Proof of Lemma \ref{lem:pA_trA}}
By the triangle inequality,
\bea
\left\| \frac{p\bm A}{\tr(\bm A)}  - \bm B \right\|_{\max} \le \|\bm A - \bm B\|_{\max} + \|\bm A\|_{\max}
\frac{|1-\tr(\bm A)/p|}{\tr(\bm A)/p}. \label{first_line} \nonumber
\eea
Next, observe that
 $\|\bm A\|_{\max} \le \|\bm B\|_{\max}+ \|\bm A - \bm B\|_{\max}\leq b_{\max}+1/2$,
and since $\tr(\bm B)=p$
\bea
|1-\tr(\bm A)/p| &=& \frac1p|\tr(\bm B) - \tr(\bm A)| =  \frac1p|\tr(\bm B - \bm A)|  \le \|\bm A - \bm B\|_{\max}.
    \nonumber
\eea
Hence,  $\tfrac1{\tr(\bm A)/p} \le \frac{1}{1-\|\bm A - \bm B\|_{\max}} \leq 2$. Combining these proves the lemma.


\subsection{TME with outliers}
\label{sec:estimate_mu_sigma}

Consider an $\epsilon$-contamination model, where $(1-\epsilon)n$ of the samples come from an elliptical distribution with shape matrix $\bm S_{in}$, and the remaining $\epsilon n$ from an elliptical distribution with shape matrix $\bm S_{out}$. We conjecture that under suitable assumptions, for $p,n\gg 1$, the weights of the TME concentrate around two values, $w_{in}$ and $w_{out}$, for the inliers and outliers, respectively.

For our procedure to select the inliers, we further assume that the inlier weights are approximately Gaussian distributed around $w_{in}$ with an unknown standard deviation $\sigma_{in}$. To estimate $w_{in}$ and \(\sigma_{in}\) we compute a non-parametric density estimate $\hat f(w)$ of all \(n\) weights (using MATLAB's \texttt{ksdensity} procedure). Then $w_{in} = \arg\max_w \hat f(w)$ is the weight with highest estimated density. Next, for some $r$ we find the largest interval $[w_L,w_R]$ around $w_{in}$ so that for $w\in[w_L,w_R]$ we have  $\hat f(w)\ge r \max \hat f(w)=r \hat f(w_{in})$. Then, given our assumption that the weights are Gaussian distributed,  $\sigma_{in} = \frac12(w_R-w_L)/\sqrt{-2\log(r)}$. In our simulations we used $r=0.7$, which
worked well across all different contamination levels.

Of course, one might obtain improved estimates of these quantities, as well as the unknown $\epsilon$, for example by fitting a mixture of two Gaussians to the vector of weights. However, for our illustrative example, we opted for the above simpler procedure.

\bibliographystyle{imsart-nameyear}
\bibliography{thresh}

\printindex

\end{document}